\documentclass{amsart}

\usepackage{amsmath}
\usepackage{amsthm}
\usepackage{float}
\usepackage{upgreek}
\usepackage{verbatim}
\usepackage{hyperref}
\usepackage{mathrsfs}
\usepackage{bm}
\usepackage{dsfont}
\usepackage{amsfonts}
\usepackage{amssymb}
\usepackage{bbm}
\usepackage{csquotes}
\usepackage{stmaryrd}
\usepackage{mathtools}
\usepackage[font=small,labelfont=bf]{caption}
\usepackage{subcaption}
\usepackage{algorithm}
\usepackage{algpseudocode}

\usepackage[margin=1 in]{geometry}
\usepackage{seqsplit}
\usepackage{xstring}

\usepackage{enumitem}
\usepackage{cleveref}
\usepackage{graphicx}
\usepackage[dvipsnames]{xcolor}
\usepackage{soul}
\sethlcolor{yellow}

\usepackage{empheq}
\usepackage{tikz}
\usetikzlibrary{shapes.geometric, arrows, calc}

\numberwithin{equation}{section}

\newtheorem{assumption}{Assumption}

\newtheorem{theorem}{Theorem}
\newtheorem{lemma}[theorem]{Lemma}

\newtheorem{remark}{Remark}

\DeclareMathOperator*{\argmax}{arg\,max}
\DeclareMathOperator*{\argmin}{arg\,min}

\newcommand{\A}{\mathcal{A}}
\newcommand{\R}{\mathbb R}
\newcommand{\PP}{\mathbb P}
\newcommand{\N}{\mathbb N}
\newcommand{\I}{\mathbb I}

\newcommand{\E}{\mathbb E}

\newcommand{\e}{\varepsilon}
\newcommand{\F}{\mathcal{F}}

\newcommand{\NN}{\mathcal{N}}

\newcommand{\ZZ}{\mathcal{Z}}

\newcommand{\YY}{\mathcal{Y}}

\newcommand{\revised}{\textcolor{black}} 

\title{\vspace{-1.5cm}A neural network approach to high-dimensional optimal switching problems with jumps in energy markets\footnote{To appear in SIAM Journal on Financial Mathematics.}}

\author{Erhan Bayraktar}
\address{Department of Mathematics, University of Michigan, 
Ann Arbor, MI}
\email{erhan@umich.edu}

\author{Asaf Cohen}
\address{Department of Mathematics, University of Michigan, 
Ann Arbor, MI}
\email{asafc@umich.edu}

\author{April Nellis}
\address{Department of Mathematics, University of Michigan, 
Ann Arbor, MI}
\email{nellisa@umich.edu}

\begin{document}

\maketitle 

\begin{abstract}
We develop a backward-in-time machine learning algorithm that uses a sequence of neural networks to solve optimal switching problems in energy production, where electricity and fossil fuel prices are subject to stochastic jumps. We then apply this algorithm to a variety of energy scheduling problems, including novel high-dimensional energy production problems. Our experimental results demonstrate that the algorithm performs with accuracy and experiences linear to sub-linear slowdowns as dimension increases, demonstrating the value of the algorithm for solving high-dimensional switching problem.

\end{abstract}


\section{Introduction}
\label{sec:intro}
Energy production and energy markets play a large role in the modern economy and as such, it is beneficial to both producers and consumers for electricity production to be optimized. 
Energy producers, in particular, desire to operate efficiently despite the inherent volatility of both electricity demand and the availability of various fuels.
Determining the correct operating strategy for an energy production facility therefore requires dynamic adjustment as the underlying drivers of price and profit fluctuate stochastically with supply and demand.
Recent supply-chain issues in global markets have further underlined the volatility of prices and the need for flexible optimization methods which allow producers to dynamically adapt to changes in the energy markets.

There are multiple perspectives from which to approach problems related to energy production and pricing. 
In the case where a model includes only a single power generation facility, this facility is considered a price-taker, and its production decisions have little impact on the overall flow of electricity supply and demand. The facility's only goal is to maximize its own profit, as it is not the sole electricity producer in its region. 
To this end, the facility is able to alter its own production capacity in response to exogenous outside factors. 
However, we can also consider a situation in which an agent oversees multiple power generation facilities, and has the option to bring them online or remove them. 
Each of these facilities is fueled by one of a selection of fuel sources, ranging from coal to solar energy. 
The larger scale of this operation makes this agent a price-setter, and so investment decisions affect both electricity spot prices and their own profits. 
In this case, penalties could also be incurred for failing to satisfy electricity demand. 
Our focus will be on the former case, but our algorithm could easily be extended to other situations. 

These situations can be modeled as optimal switching problems, and in our paper we present a machine learning algorithm that is able to solve optimal switching problems of higher dimensions than previously studied, allowing us to consider a wider selection of fuel sources than in existing literature.
Such energy production switching problems consist of a stochastic state process (such as exogenous electricity demand and fuel prices) which drives an objective function. At discrete ``switching times" a production decision is chosen from a discrete set of possible ``modes" of production (which can model factors like capacity level or fuel type).
The controller switches between modes based on the current value of the state variable, but must pay a penalty for such switches (usually monetary, reflecting resource redirection). 
The class of optimal switching problems is one that has both been investigated from an analytical perspective \cite{HamaJean2007, BayEg2010, ChasElieKhar12, ChasRich19} and applied to fields from finance \cite{LiQuZhang} to cloud computing \cite{FeinZhang2014}, but these problems remain difficult to solve numerically in higher dimensions. 
In the realm of energy markets, mathematicians have used optimal switching to model power plant scheduling \cite{CarLud08, PorchetTouziWarin2009},  electricity spot prices \cite{Aidetal14}, and run-of-river hydroelectric power generation \cite{Olofsson2021}. 
Energy storage problems \cite{Felix2012, Malyscheff, Thompson} are another popular application of optimal switching, but we focus on scheduling and production problems in our current work.

Various approaches have been taken to avoid a grid-based method, as grids are very susceptible to the so-called ``curse of dimensionality", including many Monte Carlo-based methods like \cite{Warin2018,Aidetal14}. However, such probabilistic approaches are also limited in the dimension they can handle, as most rely on regression over a number of basis functions that grows quickly with the dimension of the state space. 
In recent years, the applications of machine learning to mathematical problems has become more and more common, many inspired by seminal works such as \cite{HanJentzenE}, which trains a neural network to minimize the global error associated with the backward stochastic differential equation (BSDE) representation of certain classes of partial differential equations (PDEs). 
Expanding upon this work, neural networks have been found to accurately estimate the solutions of a variety of partial differential equations of varying complexities when used in different configurations, as in \cite{PhamML19, Sirignano18, BeckerCheriditoJentzen20, GermainPhamWarin22}. 
\revised{In addition, the early paper \cite{Barrera2006} utilized neural networks to solve for an optimal gas consumption strategy under uncertainty.}
It follows that such neural network-based methods can be extended to solve optimal switching problems. 
Our algorithm draws upon the neural-network-based deep backward dynamic programming approach introduced in \cite{PhamML19} and extends it to situations where the reflection boundary is no longer a known function, like the payoff of an American option.
Instead, the reflection boundary becomes dependent on the optimal control decision at the given point in time. 
We also introduce jumps in the state process, which change the associated formulation from a partial differential equation to a partial integro-differential equation (PIDE). 
These jumps are incorporated into the model to better simulate the volatility inherent in electricity and fossil fuel markets.
\revised{The recent work \cite{Gnoatto2022} extends \cite{HanJentzenE} to a setting with jumps, and \cite{FreyKock2022} applies neural networks to PIDEs that arise in insurance mathematics. }
In this paper, we extend \cite{PhamML19} to handle both jumps and switches in a wider range of problems.
This algorithm is able to handle high-dimensional problems well because the time needed for artificial neural network computations grows only linearly in the dimension of the state variable and suffers only minimal slowdowns as the dimension increases, as demonstrated in \Cref{sec:numerics}. Our code can be found at \url{https://github.com/april-nellis/osj}.

In \Cref{sec:switching}, we introduce the general stochastic model of an optimal switching problem. 
In \Cref{sec:algorithm}, we provide some background on neural networks and detail the proposed machine learning algorithm. 
In \Cref{sec:numerics} we discuss numerical examples of energy scheduling and capacity investment, and demonstrate the high-dimensional abilities of our algorithm\footnote{All calculations in this paper were performed on a 10-core CPU, 16-core GPU 2021 Macbook Pro with M1 Pro chip, without using GPU acceleration.}. 
In \Cref{sec:proof}, we verify the convergence of the neural networks in our proposed algorithm to the true value functions. 	


\section{Stochastic Model}
\label{sec:switching}
\subsection{Setup}
\label{sec:setup}
The goal of our paper is to numerically solve high-dimensional optimal switching problems related to energy production.
Consider a filtered probability space $(\Omega, \F, \{\F_t\}_t, \PP)$ satisfying the usual conditions and supporting a $d$-dimensional Wiener process $W$ and a one-dimensional Poisson random measure $\NN(de, ds)$ with intensity measure $\nu(de) ds$, where $\int_{\R^d} \nu(de) = \lambda\geq 0$.
 Consider further a $d$-dimensional jump-diffusion process, given by
\begin{equation}
\label{eq:state}
X_t = x_0 + \int_0^t b(X_s) ds + \int_0^t \sigma(X_s) dW_s + \int_0^t \int_{\R^d} \beta(X_{s^-}, e) \NN(de,ds),\ t \in [0, T], x_0 \in \R^d.
\end{equation}
Here, $b: \R^d \to \R^d$, $\sigma:\R^d \to \R^{d \times d}$, and $\beta: \R^d \times E \to \R^d$, where $d$ is a relatively large dimension and $E \subseteq \R^d$.
\begin{assumption}
\label{assm_beta_etc}
We assume that 
\begin{enumerate}
\item The functions $b$, $\sigma$, and $\beta$ are Lipschitz, and $\beta$ is a measurable map such that there exists $K > 0$ for which
\begin{equation*}
\sup_{\xi \in E} |\beta(0, \xi)| \leq K \text{   and   } \sup_{\xi \in E} |\beta(x, \xi) - \beta(x', \xi)| \leq K|x - x'|,\ \forall x, x' \in \R^d.
\end{equation*}
\item The function $\beta(x, \xi)$ has Jacobian such that $\nabla \beta(x, \xi) + I_d$ is invertible with a bounded inverse. 
\end{enumerate}
\end{assumption}

\begin{remark}
From the appendices of \cite{BouchardElieApprox}, the conditions in \cref{assm_beta_etc} imply the existence of a unique adapted solution $X_t$ to \cref{eq:state}.
\end{remark}

This stochastic process drives an optimal switching problem which we will solve using a series of artificial neural networks. Variations on this problem have been studied in previous theoretical papers such as \cite{HamZhang10} and \cite{ElieKharroubi14}, and can be summarized as trying to find the optimal choice of control $\bm{a} = \{(\tau_k, \alpha_k)\}_{k \in \N}$, where $\alpha_k \in \I =:\{1,\ldots, I\}$ is the regime/mode which is selected at switching time $\tau_k$, such that $\alpha_k$ is $\F_{\tau_k}$-measurable, for any $k \in \N$. We set $\tau_0 = 0$ to denote that a switch is allowed as soon as the process begins. The initial mode of the system, $i$, is therefore denoted by $\alpha_{-1} = i$. The control process $(a_s)_{s \in [0,T]}$ associated with $\bm{a}$ is denoted by:
\begin{equation*}
a_s = \sum_k \alpha_k \mathbf{1}_{\{\tau_k \leq s < \tau_{k+1}\}}.
\end{equation*}
It represents the current mode of the system and the set of such strategies is given by $\A$. The set of \textit{admissible strategies} is defined as all strategies fitting the above description which contain only a finite, though potentially random, number of switches, and is denoted $\A$. The set of admissible strategies that begin in mode $i$ at initial time $t$ is denoted $\A_{t,i}$. The expected payoff function associated with the control $\bm{a} \in \A_{t,i}$ is given by
\begin{equation}
\label{eq:J-payoff}
  J(t, x, i, \bm{a}) := \E\left[\int_t^T f_{a_s}(s, X_s) ds + g^{a_T}(X_T) - \sum_{k \in \N\backslash\{0\}} C_{\alpha_{k-1}, \alpha_k}(X_{\tau_k}) \mathbf{1}_{\{t \leq \tau_k < T\}} \Big| X_t = x, a_t = i \right],
\end{equation}
where $f_i:\R \times \R^d \to \R $ is the running profit in mode $i$, $g^i: \R^d \to \R$ is the terminal profit if ending in mode $i$, and $C_{i,j}: \R^d \to \R$ is the cost of switching modes from $i$ to $j$ for a given value of the state variable, where $i,j \in \I$.
Here and in the sequel, $\mathbf{1}_B$ is the indicator of the event $B$, such that $\mathbf{1}_B(\omega) = 1$ if $\omega \in B$ and 0 otherwise.
We can define the initial status of the system as $\F_0 := \{X_0 = x_0, \alpha_{-1}  = i\}$. All expectations are conditioned on $\F_0$ when not otherwise specified.
\begin{assumption}
\label{assm_cost}
To discourage an optimal strategy with multiple instantaneous switches, we make the following assumptions on the switching costs. There exists $\epsilon > 0$ such that
\begin{equation*}
\begin{split}
&C_{i,j}(x) \geq \epsilon > 0,\ \forall i,j \in \I,\ \forall x \in \R^d,\\
&C_{ii}(x) \equiv 0,\ \forall i \in \I,\\
&C_{i,j}(x) + C_{j,k}(x) \geq C_{i,k}(x),\ \forall i,j,k \in \I,\ \forall x \in \R^d.
\end{split}
\end{equation*}
These assumptions are standard in optimal switching problems, encoding ``direct" switches between states, and are enforced throughout the paper. We also make the Lipschitz assumption that there exists a constant $[C]_l$ such that 
\begin{equation*}
|C_{i,j}(x_1) - C_{i,j}(x_2)| \leq [C]_l ||x_1 - x_2||,
\end{equation*}
for all $x_1, x_2 \in \R^d$ and all $i,j \in \I$.
\end{assumption}
We also make certain assumptions on the running profit and terminal profit functions throughout the paper. 
\begin{assumption}.
\label{assm_fg}
\begin{enumerate}
	\item There exists a constant $[f]_l$ such that for every $t_1, t_2 \in [0,T]$ and $x_1, x_2 \in \R^d$,
		\begin{equation*}
			|f_i(t_1, x_1) - f_i(t_2, x_2)| \leq [f]_l(|t_1 - t_2|^{1/2} + ||x_1 - x_2||),\ \forall i \in \I.
		\end{equation*}
		
	\item We assume $\max_{i \in \I} \sup_{0 \leq t \leq T} |f_i(t,0)| < \infty$ and $f_i$ is square-integrable on $[0,T]$ for all $i$ in $\I$.
	  
	\item The functions $\{g^i\}_{i \in \I}$ are Lipschitz continuous and satisfy linear growth conditions.
\end{enumerate}
\end{assumption}
The value function is given by 
\begin{equation*}
    V(t, x, i) := \sup_{\bm{a} \in \A_{t,i}} J(t, x, i, \bm{a}).
\end{equation*}
It is a standard result in control theory (see \cite{ElKaroui97, Biswas10}) that the solution to this optimization problem can be represented as a real-valued stochastic process $Y^i_t = V(t, X_t, i)$ that solves the stochastic differential equation
\revised{\begin{equation}
\label{eq:BSDE}
\begin{split}
&Y^i_t = g^i(X_T) + \int_t^T f_i(s, X_s) ds - \int_t^T (Z^i_s)^T dW_s - \int_t^T \int_{\R^d} \Delta Y^i_s(e) \tilde \NN(de, ds) + R^i_T - R_t^i,\\
&Y_t^i \geq \max_{j \neq i}\{-C_{i,j}(X_t) + Y^j_t\},\ t\in[0,T], \\
&\int_0^T ( Y^i_t - \max_{j \neq {i}}(-C_{i,j}(X_t) + Y^j_t)) dR_t^i  = 0.
\end{split}
\end{equation}
where $\tilde \NN(de, ds) := \NN(de,ds) - \nu(de) ds$ and the reflection boundary $R^i_t$ is a nondecreasing process with $R^i_0 = 0$. }
Further, the auxiliary processes $Z^i_t$ and $\Delta Y^i_t(e)$ can be defined as
\begin{align*}
&Z^i_t := \sigma^T(X_t) V_x(t, X_t, i) \in \R^d, \\
&\Delta Y^i_t(e) := V\big(t, X_{t^-} + \beta(X_{t^-}, e), i\big) - V(t, X_{t^-}, i) \in \R.
\end{align*}

The stochastic differential equations for $X_t$ and $Y_t$ comprise a system of forward-backward stochastic differential equations (FBSDEs). In addition, the continuation values associated with beginning in mode $i$ at time $t_1$ and remaining in that mode over the interval $[t_1, t_2]$ for $t_1, t_2 \in [0,T]$ can be defined via \cref{eq:BSDE} as
\begin{equation*}
\tilde Y^i_{t_1} := Y^i_{t_2} + \int_{t_1}^{t_2} f_i(s, X_s) ds - \int_{t_1}^{t_2} (Z^i_s)^T dW_s - \int_{t_1}^{t_2} \int_{\R^d} \Delta Y^i_s(e) \tilde \NN(de, ds).
\end{equation*}
Approximation of certain continuation values will play a key role in approximating the value function of interest. 
Our goal in the rest of this paper is to present an efficient algorithm for calculating $Y^i$ where $X$ is a high-dimensional state process with finite-variational jumps. In \cref{sec:algorithm}, we first provide some background on neural networks, then we present the details of the Optimal Switching with Jumps (OSJ) algorithm. 	

\section{Optimal Switching with Jumps (OSJ) Algorithm}
\label{sec:algorithm}
\subsection{Neural Network Structure}
We utilize feedforward neural networks, which are in essence a series of weighted sums of inputs composed with simple functions in such a way that unknown functions can be approximated. Training data enters the network in the first layer, and at each layer a weighted sum of the inputs is computed using the choice of parameters assigned to the nodes in that layer to create an affine function. The output of each layer is processed by an activation function before becoming the input of the next layer, and the final layer produces the desired output of the network.

For a network of depth $\delta$ with $\delta_\ell$ nodes in layer $\ell$, there are $\sum_{\ell = 0}^{\delta-1} \delta_\ell (\delta_{\ell + 1}+1) = \bar \delta$ parameters, represented as a whole as $\theta$.  This $\theta$ is chosen from all possible parameters in the parameter space $\Theta_\delta$, a compact subset of $\R^{\bar \delta}$ defined as
\begin{equation*}
	\Theta_\delta := \{\theta \in \R^{\bar \delta},\  ||\theta||_{\infty} \leq \gamma_\delta \},
\end{equation*}
where $\gamma_\delta$ is positive and chosen to be very large.
We can then define the set of neural networks that we are working with as the union over $\delta \in \N$ of all the neural networks of depth $\delta$ with $\bar \delta$ total parameters. This formulation accomplishes two things. First, the universal approximation theorem of \cite{Hornik89} asserts that this set of neural networks is dense in the set of continuous and measurable functions which map from $\R^d \to \R^s$, for any dimension $s$, and so are universally good approximators. Second, the parameter space associated with this union, $\Theta = \cup_{\delta \in \N} \Theta_\delta$, represents the set of all possible weights that can be assigned to the nodes in the neural network and is compact. Therefore, when trying to minimize the loss function associated with our problem (which will be described in the next subsection), a minimizing $\theta^*$ exists.

The network therefore ``learns" the function of interest by adjusting $\theta$ via multiple iterations of an optimization algorithm. In our work, we use the Adam optimizer \cite{adamopt} applied to a four-layer neural network with $d + 10$ nodes in each layer and $tanh$ as the chosen activation function. We fix the input dimension as $d$, and set the output dimension as $d_1 = 1 + d + 1$ because $Y^i_t \in \R, Z^i_t \in \R^d$, and $\Delta Y^i_t \in \R$. 

\subsection{Algorithm}
To perform the numerical calculations, we discretize the continuous time interval $[0,T]$ using a regular grid $\pi = \{t_n\}_{n=0}^M = \{nT/M\}_{n=0}^M$, where $T/M = \Delta t$. 
We denote the paths of the discrete approximation as $X^\pi$ and generate a large number of paths of $X^\pi$ starting from a desired initial condition $x_0$. We later use these paths as training data for the neural networks. 
At this point, we do not impose a specific approximation scheme, but require that the in-time convergence be of at least strong order 0.5 for convergence of the neural network value function in \cref{thm:main} and of strong order 1.0 to achieve an auxiliary result regarding the performance of the neural network-generated strategies given in \cref{thm:strat}. We describe the approximation schemes used in our specific examples within the expository portions of \cref{sec:ex_CL} and \cref{sec:ex_Aid}. In-depth discussion of weak- and strong-order approximations of jump-diffusion processes can be found in \cite{KloedenPlaten1999}.

\revised{We also discretize the switching times, introducing a grid $\mathfrak{R}$ where the grid spacing is of size $T/\sqrt{M}$, meaning that $|\mathfrak{R}| \sim O(M^{-1/2})$. The process is able to switch modes at time $t_n$ where $t_n \in \mathfrak{R}$, while evolving as an uncontrolled jump-diffusion process when $t_n \notin \mathfrak{R}$.}

The continuation values between time steps of the grid $\pi$ will be learned by the neural network on a mode-by-mode basis. We denote the neural network that learns the continuation value at $t_n$ in mode $i$ as $\YY^i_n(X^\pi_n, \theta^i_1)$, the neural network that learns the derivative of the value function at the same stage as $\ZZ^i_n(X^\pi_n, \theta^i_2)$, and the neural network that learns the jump sizes for $Y$ as $\Delta \YY^i_n(X^\pi_n, \theta^i_3)$. In practice, these neural networks are treated as one larger network with combined parameter vector $\theta^i = (\theta^i_1, \theta^i_2, \theta^i_3) \in \Theta = \Theta_\delta$ for some $\delta$ corresponding to the chosen architecture of the neural network. The functions generated by the optimal choice of $\theta^i$ are defined as
\revised{\begin{equation}
\label{eq:NN_funcs}
\begin{cases}
	&\tilde \YY^i_n(X^\pi_n) := \YY^i_n( X^\pi_n,\theta^{*,i}_{n,1}),\\
	&\hat \ZZ^i_n(X^\pi_n) := \ZZ^i_n( X^\pi_n, \theta^{*,i}_{n,2}),\\
	&\widehat{\Delta \YY}^i_n(X^\pi_n) := \Delta \YY^i_n(X^\pi_n, \theta^{*,i}_{n,3}).
\end{cases}
\end{equation}
The continuation values $\tilde \YY^i_n(\cdot)$ are then used to calculate the value functions
\begin{equation}
\label{eq:NN_Yhat}
\revised{\hat \YY^i_n(X^\pi_n) := \mathbf{1}_{t_n \in \mathfrak{R}} \max\big\{\tilde \YY^i_n(X^\pi_n), \max_{j\neq i}(\tilde \YY^j_n(X^\pi_n) - C_{i,j}(X^\pi_n))\big\} + \mathbf{1}_{t_n \notin \mathfrak{R}}\tilde \YY^i_n(X^\pi_n)}
\end{equation}}
The algorithm in its entirety is described in \cref{osj-alg}. 

\begin{algorithm}
\caption{OSJ Algorithm}
\label{osj-alg}
	\begin{algorithmic}[1]
	\State{Generate paths of the stochastic process $\{X^\pi_n\}_{n=0}^M$ as well as $\Delta W_n := W_{t_{n+1}} - W_{t_n}$ and $\Delta \tilde N_n :=\int_{t_n}^{t_{n+1}} \int_{\R^d} \tilde \NN(de, ds)$ for each sample path. Store as training data.}
	 \State{Train $\hat \YY^i_M \equiv g^i(x),\ \forall i \in \I$.}
	 \For{ $n = M -1, \ldots, 2, 1,0$}
	 	\For{all $i \in \I$} 
		\State{Train a neural network to find $\theta^{*,i}_n = (\theta^{*,i}_{n,1},\theta^{*,i}_{n,2},\theta^{*,i}_{n,3}) \in \Theta$ which minimizes} 
    	\begin{equation}
    	\begin{split}
    	\label{eq:loss}
    		L^i_n(\theta) =  \E\Big|&\hat \YY^i_{n+1}(X^\pi_{n+1}) - \YY^i_n(X^\pi_n, \theta_1)\\
		&+ f_i(t_n, X^\pi_n) \Delta t - \ZZ^i_n(X^\pi_n, \theta_2)\Delta W_n -\Delta \YY^i_n(X^\pi_n, \theta_3)\Delta \tilde N_n \Big|^2.
	\end{split}
	\end{equation}
	\State{Define $\tilde \YY^i_n(\cdot)$, $\hat \ZZ^i_n(\cdot)$, and $\widehat{\Delta \YY}^i_n(\cdot)$ as in \cref{eq:NN_funcs}.}
	\EndFor
	\For{all $i \in \I$}
		\State{Calculate $\hat \YY^i_n(\cdot)$ as in \cref{eq:NN_Yhat}.}
	\EndFor
	\State{The value function of interest is $V(0, x_0, i) = \hat \YY^i_0(x)$ where $X_0 = x$ and $\alpha_{-1} = i$.}
	\EndFor
	\end{algorithmic}
\end{algorithm}

The switching strategy arising from this algorithm is denoted $ \bm{a^{NN,M}}$ when the number of steps is chosen to be $M$. This strategy is a function of the value of the state variable and starting mode, such that the optimal mode at time $t_n$ is 
\begin{equation*}
\alpha^{NN}_n := \argmax_{j \in \I} \big(\tilde \YY^j_n(\cdot) - C_{\alpha^{NN}_{n-1},j}(\cdot)\big),
\end{equation*}
 where the optimal mode at time $t_{n-1}$ is $\alpha^{NN}_{n-1} \in\I$ and $\alpha^{NN}_{-1} = \alpha_{-1} = i$, the starting mode.
 
 \begin{remark}
\label{rmk:disc}
For \cref{thm:strat} we require that the discrete approximation of $(X_t)_{t \in [0,T]}$ is of strong order 1.0 (instead of strong order 0.5 which is sufficient for \cref{thm:main}). This excludes the standard Euler--Maryama discretization strategy, but there are a variety of other options. One good choice is a jump-adapted strong order 1.0 discretization scheme, first described in \cite{Platen1982}, and other efficient schemes are described in \cite{maghsoodiMSE} and \cite{maghsoodiExact}. A survey of strong approximations of jump-diffusions is given in \cite{PlatenBL2007} \revised{and expanded upon in chapters 6 and 8 of \cite{PlatenBL2010}}. Since the path generation is done at the start of the algorithm and the paths can be stored for re-use, the downside of increased complexity is outweighed by the benefit of improved convergence rates. In addition, if paths can be simulated exactly (as is the case of the example in \cref{sec:ex_CL}), the discretization error in $X$ can be avoided altogether, though there remains discretization error from the discrete switching grid.
\end{remark}

\subsection{Convergence}
The natural questions are (1) whether this algorithm will train a neural network that converges to the correct value function of interest, and (2) how the switching strategy generated by this algorithm performs in comparison to the true optimal switching strategy. We first focus on the former question.

The mode-wise maximum of the errors between the FBSDE system $(Y^i_t, Z^i_t, \Delta Y^i_t)$ and the functions learned by the neural network $(\hat\YY^i_n, \hat\ZZ^i_n, \widehat{\Delta \YY}^i_n)$, can be represented by
\begin{equation}
\begin{split}
\label{eq:err_list}
\mathcal{E}[(\hat\YY, \hat\ZZ, \widehat{\Delta \YY})), (Y, Z, \Delta Y)]:= &\max_{n = 0,1,\ldots,M} \E\Big[\max_{i \in \I} |Y^i_{t_n} - \hat \YY^i_n(X^\pi_n)|^2\Big]\\
&+ \sum_{n=0}^{M-1} \int_{t_n}^{t_{n+1}} \max_{i \in \I} \E ||Z^i_t - \hat\ZZ^i_n(X^\pi_n)||^2 dt\\
&+  \sum_{n=0}^{M-1} \int_{t_n}^{t_{n+1}} \max_{i \in \I} \E \left|\int_{\R^d} \Delta Y^i_t(e) \nu(de) - \widehat{\Delta \YY}^i_n(X^\pi_n)\right|^2 dt.
\end{split}
\end{equation} 

We also define auxiliary functions which are integral to the upcoming error definition and to the proofs of convergence themselves. These functions are formulated in a similar manner to that of the discretization of the FBSDEs presented in \cite{BouchardElieApprox} and are defined as
\begin{align}
& \hat y^i_n(X^\pi_n) := \E\left[\hat \YY^i_{n+1}(X^\pi_{n+1})|\F_n\right] + f_i(t_n, X^\pi_n) \Delta t, \label{eq:hat_y}\\
&\hat z^i_n(X^\pi_n) := \frac{1}{\Delta t} \E\left[\hat \YY^i_{n+1}(X^\pi_{n+1}) \Delta W_n|\F_n\right], \label{eq:hat_z}\\
&\hat u^i_n(X^\pi_n) :=  \frac{1}{ \lambda \Delta t}\E\left[\hat \YY^i_{n+1}(X^\pi_{n+1})\Delta \tilde N_n|\F_n\right] \label{eq:hat_u},
\end{align}
where $\F_n := \F_{t_n}$ for convenience. 
According to the martingale representation theorem of \cite{TangLi}, there exist processes $(z^i_s)_{t_n < s \leq t_{n+1}}$ and $(u^i_s)_{t_n < s \leq t_{n+1}}$ such that 
\begin{equation}
\label{eq:loss-bsde}
 \hat \YY^i_{n+1}(X^\pi_{n+1})= \hat y^i_n(X^\pi_n) - f_i(t_n, X^\pi_n) \Delta t +  \int_{t_n}^{t_{n+1}} z^i_s d W_s +  \int_{t_n}^{t_{n+1}} \int_{\R^d} u^i_s(e) \tilde \NN(de, ds).
 \end{equation}
Using the Markov property of the processes and It\^o isometry, we can also see that the following relationships hold:
\begin{align}
\hat z^i_n(X^\pi_n) &= \frac{1}{\Delta t} \E\left[\int_{t_n}^{t_{n+1}} z^i_s ds | \F_n \right], \label{eq:z_hat}\\
 \hat u^i_n(X^\pi_n)  &= \frac{1}{\lambda \Delta t}\E\left[ \int_{t_n}^{t_{n+1}} \int_{\R^d} u^i_s(e) \nu(de)\ ds| \F_n \right]. \label{eq:u_hat}
\end{align}
Using these quantities, we define the set of neural network approximation errors as
\begin{equation}
\begin{split}
\label{eq:eps}
&\e^y_n := \sum_{i\in \I} \left( \inf_{\theta_1 \in \Theta} \E|\hat y^i_n(X^\pi_n) - \YY^i_n(X^\pi_n, \theta_1)|^2\right), \\
&\e^z_n :=\sum_{i\in \I} \left( \inf_{\theta_2 \in \Theta} \E ||\hat z^i_n(X^\pi_n) - \ZZ^i_{n}(X^\pi_n, \theta_2)||^2 \right), \\
&\e^u_n := \sum_{i\in \I} \left( \inf_{\theta_3 \in \Theta} \E |\hat u^i_n(X^\pi_n) -\Delta \YY^i_{n}(X^\pi_n, \theta_3)|^2 \right).
\end{split}
\end{equation}
These errors converge to 0 as the number of parameters in the neural network increases to infinity. More precisely, we can say that as $\delta \to \infty$ and correspondingly $|\Theta_\delta| \to \infty$, then $\e^y_n, \e^z_n, \e^y_n$ converge to 0. This follows from the Universal Approximation Theorem of \cite{Hornik89}. Therefore, we can use \cref{eq:eps} to express our convergence results as a function of the neural network approximation errors and the time discretization $M$.

\begin{theorem}[Convergence of OSJ Value Function]
\label{thm:main}
The error between the machine learning-generated functions $(\hat\YY^i, \hat\ZZ^i, \widehat{\Delta \YY}^i)$, and the functions associated with the true optimal switching problem $(Y^i, Z^i, \Delta Y^i)$ (as defined in \cref{eq:err_list}), is bounded by
\begin{equation}
\label{eq:thm}
\mathcal{E}[(\hat\YY, \hat\ZZ, \widehat{\Delta \YY})), (Y, Z, \Delta Y)] \leq C\sum_{n=0}^{M-1} \big(M\e^y_n + \e^z_n + \e^u_n\big) + \revised{C^{\e} \left(M^{-1/2 + \e}\right)},
\end{equation}
where $\e^y_n, \e^z_n, \e^u_n$ are defined in \cref{eq:eps}. \revised{In addition, $\e>0$ is chosen to be arbitrarily small and $C^\e$ has an inverse relationship with $\e$. The formulation of this error term is further discussed in \cref{sec:discretization}.}
\end{theorem}

\begin{remark}
\label{rmk:bigO}
In general, we write $f(x) = O(g(x))$ if there exists some $C > 0$, dependent only on the specific formulation of the problem and not on the size of the discretized grid, such that $f(x) \leq C g(x)$ as $x \to 0$.
\end{remark}

\begin{remark}
\revised{Note that our convergence scheme recovers a similar error to Theorem 4.1 of \cite{PhamML19} which we extended, with the caveat that the discrete approximation of the multidimensional reflected BSDE \cref{eq:BSDE} has size $C^\e M^{-1/2-\e}$ instead of $O(M^{-1})$ for their one-dimensional BSDE.}
\end{remark}

\begin{remark}
The quantity $M\e^y_n + \e^z_n + \e^u_n$ will reappear multiple times as we proceed through the paper. In the future, for convenience we will define
\begin{equation*}
\bm{\e^M_n} := M\e^y_n + \e^z_n + \e^u_n,
\end{equation*}
and at times we will use this simpler quantity in place of the sum of the neural network errors.
\end{remark}

We also wish to evaluate the performance of $ \bm{a^{NN,M}}$ (the switching strategy generated by the machine-learning algorithm) as $M$ increases to infinity by looking at $J(0,x_0, i,  \bm{a^{NN,M}})$ as defined in \cref{eq:J-payoff}. To do this, we introduce some additional conditions which are not necessary for the results given in \cref{thm:main}. 
\begin{theorem}[Convergence of OSJ Switching Strategy]
\label{thm:strat}
Assume that the discrete approximation of $(X_t)_{t \in [0,T]}$ is of strong order 1.0 at minimum.  Then, the mode-wise maximum of the errors between the neural network switching strategy's expected payoff $J(0,x_0, i,  \bm{a^{NN,M}})$ and the true value function $V(0,x_0,i)$ is bounded by
\begin{equation}
\label{eq:thm2}
\begin{split}
\max_{i \in \I} |V(0, x_0, i)  - J(0, x_0, i,  \bm{a^{NN,M}}) | \leq& \revised{C^\e(M^{-1/2 + \e})} + C_2\sum_{n=0}^{M-1}\left[\frac{\bm{\e^M_n}}{M} + \bm{\e^M_n}\right]\\
&+O\left( \frac{1}{M^2}\right)\sqrt{O(1) + C_5\sum_{n=0}^{M-1}\Big[\big(\bm{\e^M_n}/M\big)^{1/2} + \bm{\e^M_n}\Big]}.
\end{split}
\end{equation}
\end{theorem}

\begin{remark}
\label{rmk:NN}
From \cite{Hornik89} it is known that $\e^y_n, \e^z_n, \e^u_n \to 0$ as $|\Theta_\delta| \to \infty$. Therefore, for every $M$ there exists a neural network configuration such that all of these errors can be made arbitrarily small. This implies that for each choice of $M$, $\bm{\e^M_n}$ can also be made arbitrarily small. Specifically, for any $\zeta > 0$, there exists a neural network configuration such that $\bm{\e^M_n}$ is on the order of $O(M^{-1 - \zeta})$ for all $n \in \{0,\ldots, M-1\}$. This will ensure the convergence of the results shown in this section.
\end{remark}

\begin{remark}
\revised{We wish to make a small comment on the use of the expectation in our convergence results despite calculating an empirical mean in our algorithm. We choose the number of simulated paths to be on the order of $M^2$ to ensure that the error between the empirical mean and analytical expectation is reasonably small. Furthermore, if exact simulations can be used, this error does not depend on the time discretization. We refer to Theorem 6.2 of \cite{BouchardTouzi2004} for further reading on this distinction, but omit this error elsewhere in our paper for simplicity.}
\end{remark}

The full proofs of \cref{thm:main} and \cref{thm:strat} are provided in \Cref{sec:proof}.


\section{Numerical Examples}
\label{sec:numerics}
\subsection{Optimal Asset Scheduling (Carmona and Ludkovski) \cite{CarLud08}}
\label{sec:ex_CL}
We first numerically evaluate the OSJ algorithm by implementing an optimal switching problem first discussed in \cite{CarLud08}. The problem models a power plant that converts natural gas to electricity, where investors are able to purchase three-month lease contracts. We wish to know how to price these contracts, which is equivalent to calculating the expected maximum profit of the plant over the time horizon under no-arbitrage assumptions.

The price process informing the power plant scheduling decisions has the form of an exponential Ornstein-Uhlenbeck process with jumps given by
\begin{equation*}
dX_t = X_t\left[\kappa(\mu - \log(X_t))dt + \Sigma dW_t + \int_{\R^d} (\exp(e)-1) \NN(de, dt)\mathbf{e_1}\right]
\end{equation*}
where $\kappa, \mu, dW_t \in \R^d$ and $\Sigma \in \R^{d \times d}$. The jumps are driven by $\NN(de, dt)$, a Poisson random measure with intensity measure $\nu(de) = \lambda \mu(e) de$, where $\lambda$ is a constant and $\mu(e)$ is the probability density function of $e$.
The standard basis vector $\mathbf{e_1} = (1, 0, \ldots, 0)^T$ is used to denote that only the first dimension (electricity price) experiences a jump. 

We use a jump-adapted exact scheme to model the process, where jump times are simulated and added to $\pi$ to form a new grid, $\pi'$, with time steps $\{\tau_m\}_{m=0}^{M'}$ where $M' \geq M$. If we define $\tilde X_t = \log(X_t)$, then in between grid points in the jump-adapted scheme, the process evolves as a pure diffusion according to
\begin{equation*}
d \tilde X_t = \left[\kappa(\mu - \tilde X_t) - \frac{1}{2} \text{Tr}(\Sigma \Sigma^T) \right]dt + \Sigma dW_t, t \in (\tau_m, \tau_{m+1}),
\end{equation*}
and the exact solution of this Ornstein-Uhlenbeck process is well-known. The sizes of the randomly-occurring jumps is simulated according to the distribution of the exponential random variable $e$.

\subsubsection{Two-dimensional Example}
\label{sec:CL_2dim}
The price processes for the electricity sold by the plant and the natural gas bought by the plant are given by the following stochastic processes: 
\begin{align*}
    & dP_t = P_t \left\{5(\log(50) - \log P_t) dt + 0.5 dW^1_t + \int_{\R} (\exp(e)-1) \NN(de, dt) \right\}, \\
    & dG_t = G_t \{2(\log(6) -\log G_t) dt + 0.4(0.8 dW^1_t + 0.6 dW^2_t)\},
\end{align*}
where $e$ follow an exponential distribution where the mean jump size is 1/10 = 10\% (and $\mu$ is set acccordingly), and the Poisson random measure has $\lambda = 8$, meaning that the average number of jumps is 8 per year.
These values are chosen to be consistent with those in \cite{CarLud08}.
As previously discussed, the power plant can be leased for three-month intervals. Decisions regarding production capacity can be made twice daily, so there are 180 operational decisions to be made over the lifetime of the contract. As in \cite{CarLud08}, the switching costs are purely dependent on the price of natural gas, as we assume that a certain amount of fuel is burned when altering the operating state of the plant. In this situation, we define 
\begin{equation}
\label{eq:switching_cost}
C_{i,j} = \begin{cases} 0, & i = j,  \forall i, j \in \mathcal{I}, \\ 0.01G_t + 0.001, & i \neq j, \forall i, j \in \mathcal{I}, \end{cases}
\end{equation}
which satisfies the assumptions in \cref{assm_cost}. For our example, the power plant is assumed to begin the contract period in a dormant state, after which the controller may choose whether to run the plant at full capacity (mode 3), half capacity (mode 2), or turn it off temporarily (mode 1). Each mode has an associated running profit, described by
\begin{equation}
\label{eq:psi}
\begin{split}
    &f_1(P_t, G_t) = -1, \\
    & f_2(P_t, G_t) = 0.438(P_t - 7.5G_t) - 1.1, \\
    & f_3(P_t, G_t) = 0.876(P_t - 10G_t) - 1.2,
    \end{split}
\end{equation}
where the total capacity of the plant is 876 MW and the heat rates are 7.5 MMBtu/MWh when running at half capacity and 10 MMBtu/MWh when running at full capacity. 

While this example is only two-dimensional, we implement it to verify the accuracy of our OSJ algorithm on an example that is numerically tractable using previous probabilistic methods. We have made some changes to both the operational setup and the price process evolution in comparison to \cite{CarLud08}, so we cannot directly compare with the results obtained in that paper. However, we are able to compare the results of our implementation with the results of our implementation of their Longstaff-Schwartz probabilistic algorithm, and find that the discrepancy is very low between the methods. The results of our numerical experiments are displayed in \cref{tab:CL_results}. The visualizations of the switching strategies can be found in \cref{fig:cl-2}, and examination shows that the strategies make sense in the context of the problem. For higher natural gas costs and lower electricity prices, the plant will have less incentive to produce large amounts of electricity, since the profits of selling electricity will not fully cover the input costs. 
\begin{table}
\centering
\begin{tabular}{|c|c|c|c|}
\hline
$x_0 = [50,6]$ & OSJ & LS & \revised{Difference} \\
\hline
$V(0, x_0, 1)$ & 0.633266 & 0.63232 & 0.01\%  \\
\hline
$V(0, x_0, 2)$ & 0.69657 & 0.69683 & 0.0004\% \\
\hline
$V(0, x_0, 3)$ &0.621001 & 0.62914 & 1.29\% \\
\hline
\end{tabular}
\caption{This table compares the results of the proposed OSJ algorithm with our implementation of the Longstaff-Schwarz (LS) algorithm presented in \cite{CarLud08} to verify the accuracy of the results. 100,000 paths of $\{X^\pi_n\}_{n=0}^{180}$ were generated. \revised{The neural networks were trained with a learning rate of 0.001 over 20 epochs, and had a runtime of of 13.73 minutes.}}
\label{tab:CL_results}
\end{table}

\begin{figure}
\centering
\begin{subfigure}{0.45\textwidth}
        \includegraphics[width = \textwidth]{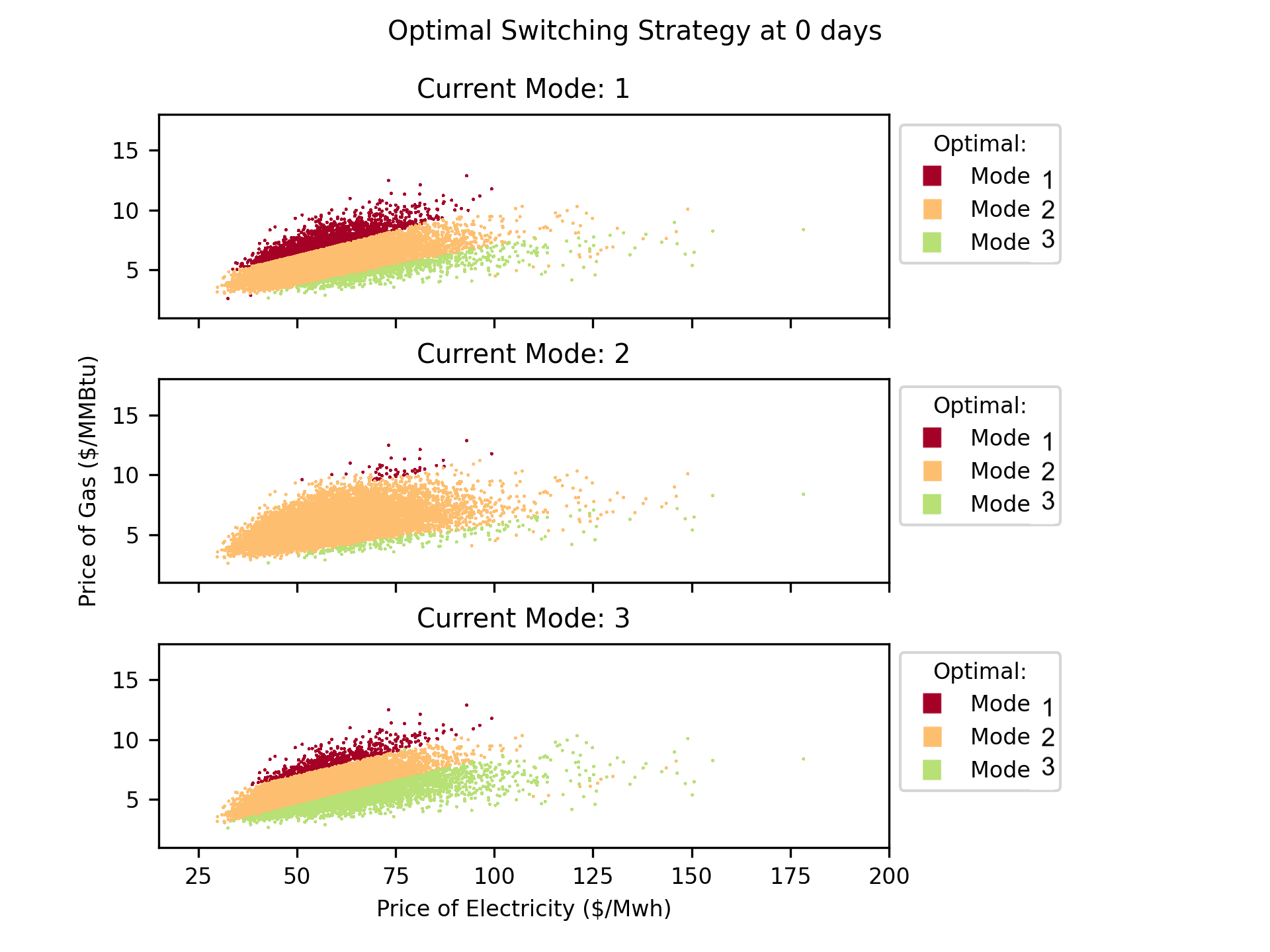}
\end{subfigure}
\begin{subfigure}{0.45\textwidth}
        \includegraphics[width = \textwidth]{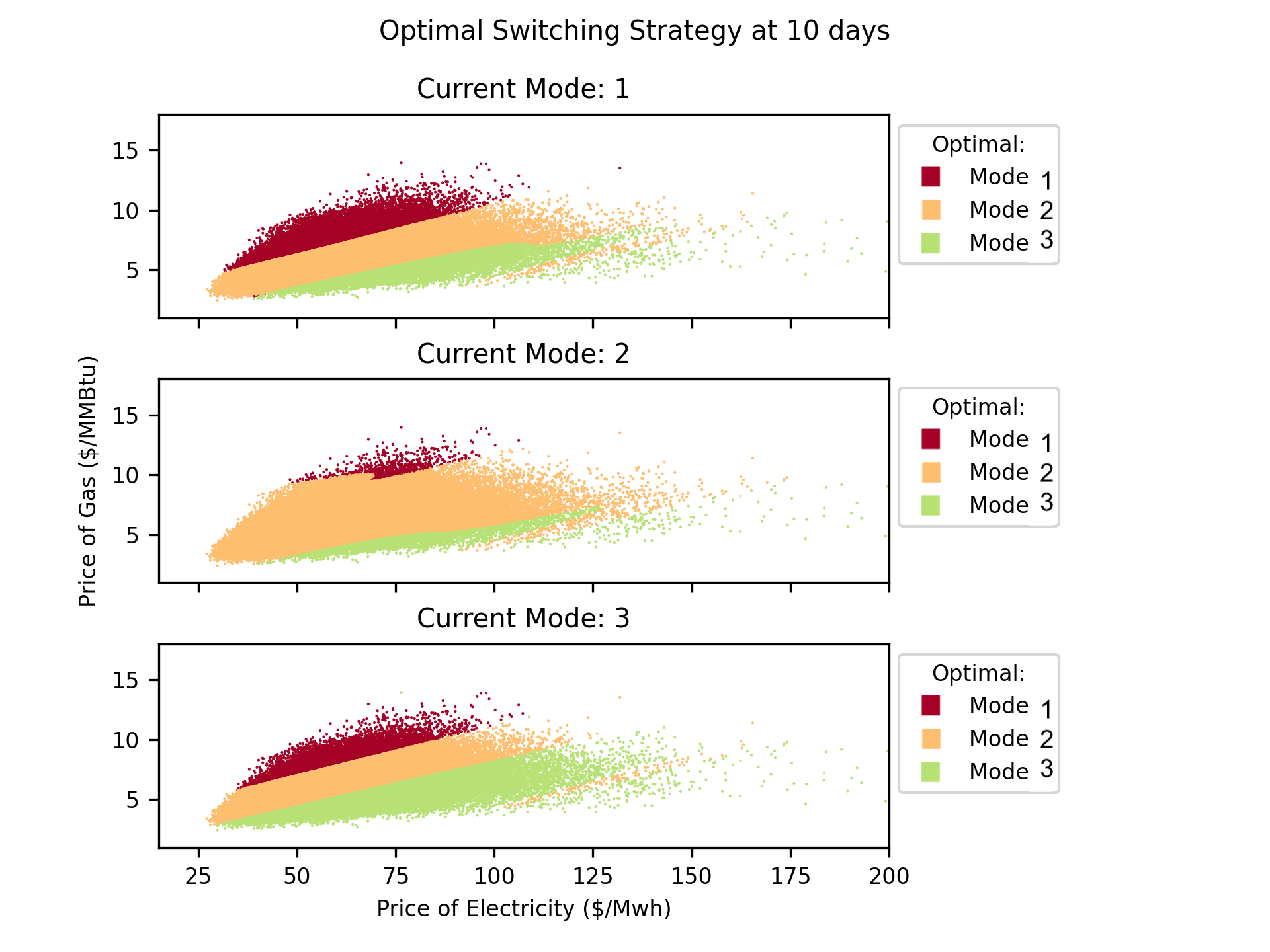}
\end{subfigure}
\begin{subfigure}{0.45\textwidth}
        \includegraphics[width = \textwidth]{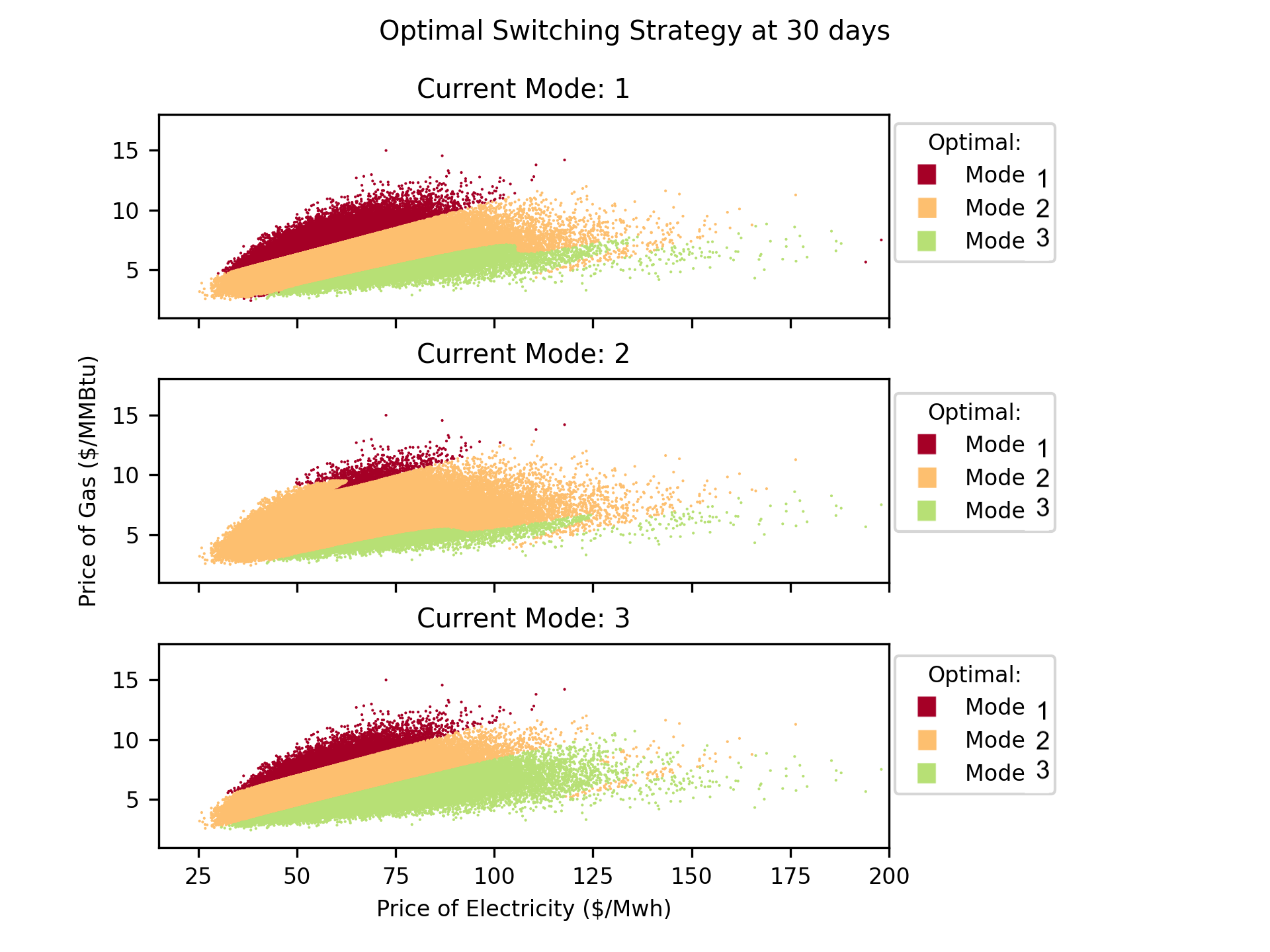}
\end{subfigure}
\begin{subfigure}{0.45\textwidth}
        \includegraphics[width = \textwidth]{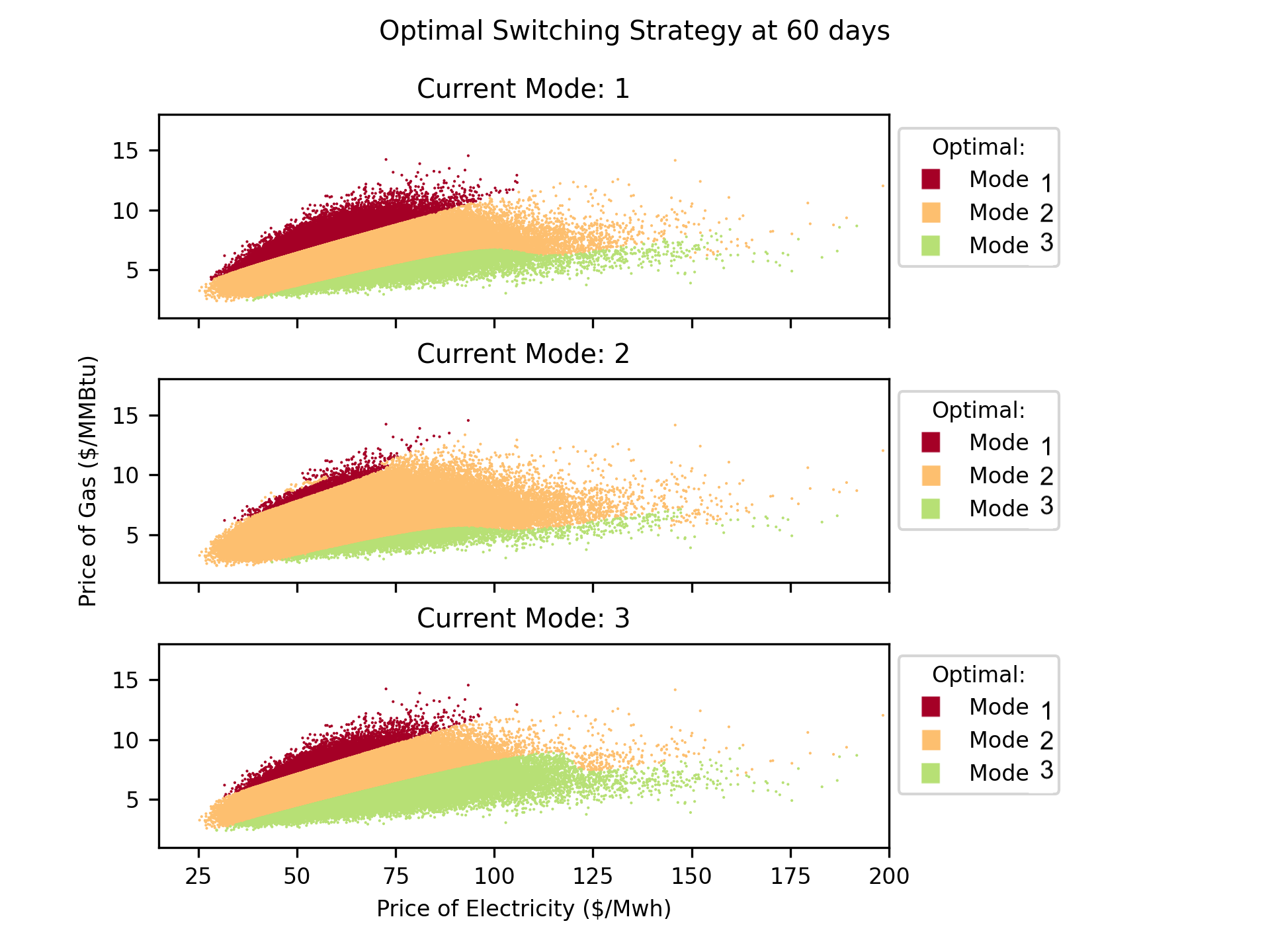}
\end{subfigure}
\begin{subfigure}{0.45\textwidth}
        \includegraphics[width = \textwidth]{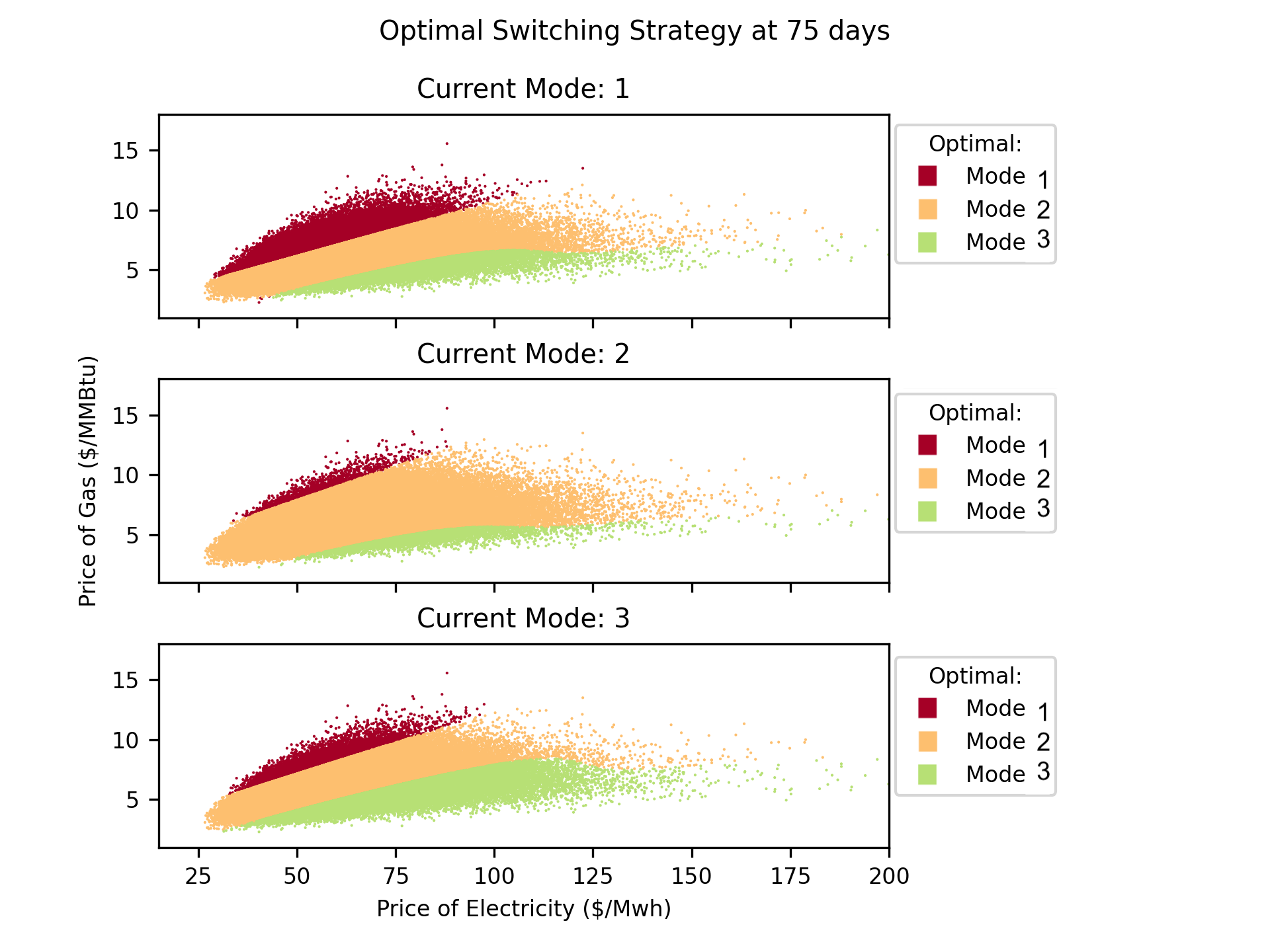}
\end{subfigure}
\begin{subfigure}{0.45\textwidth}
        \includegraphics[width = \textwidth]{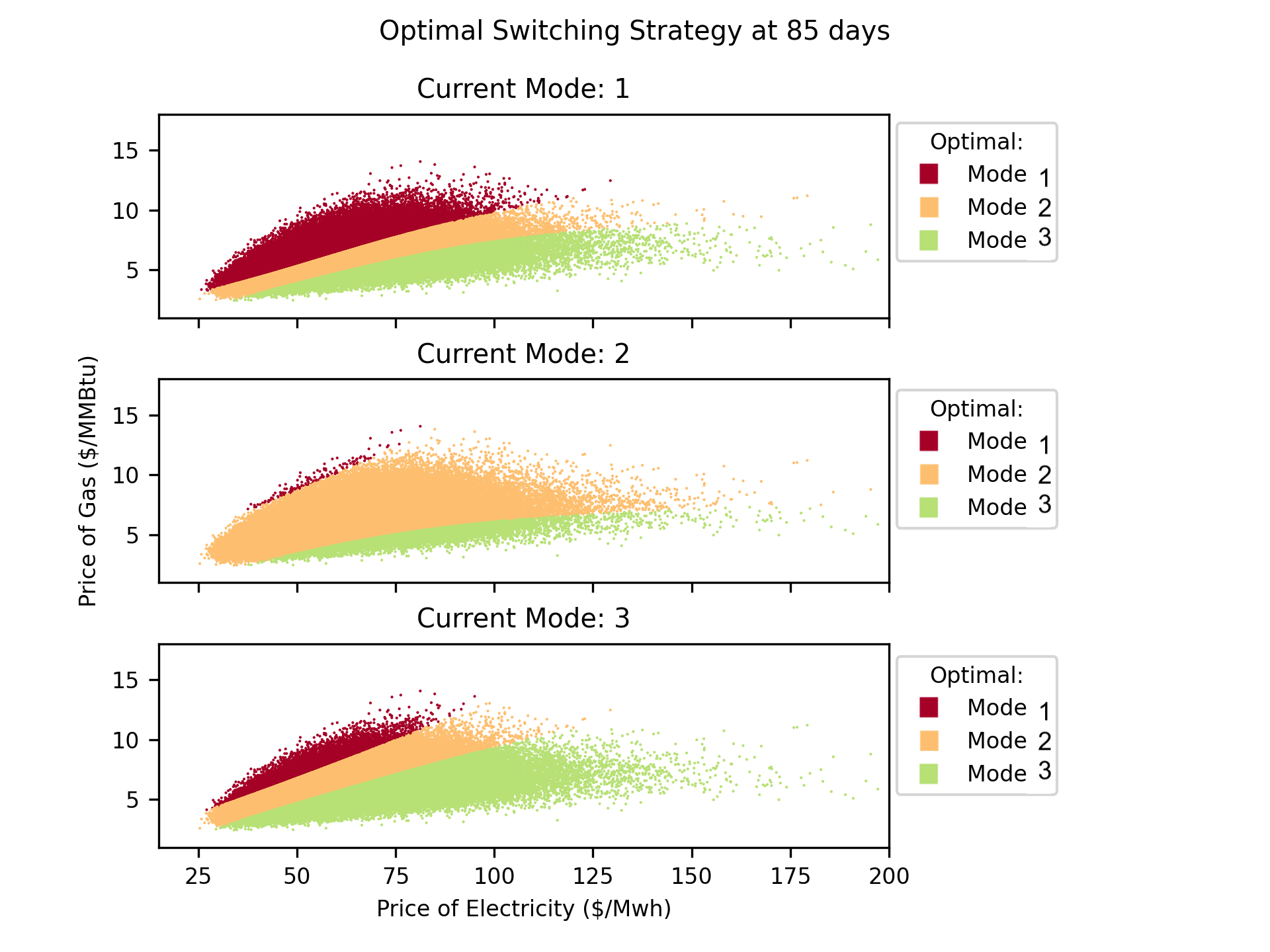}
\end{subfigure}
\begin{subfigure}{0.45\textwidth}
        \includegraphics[width = \textwidth]{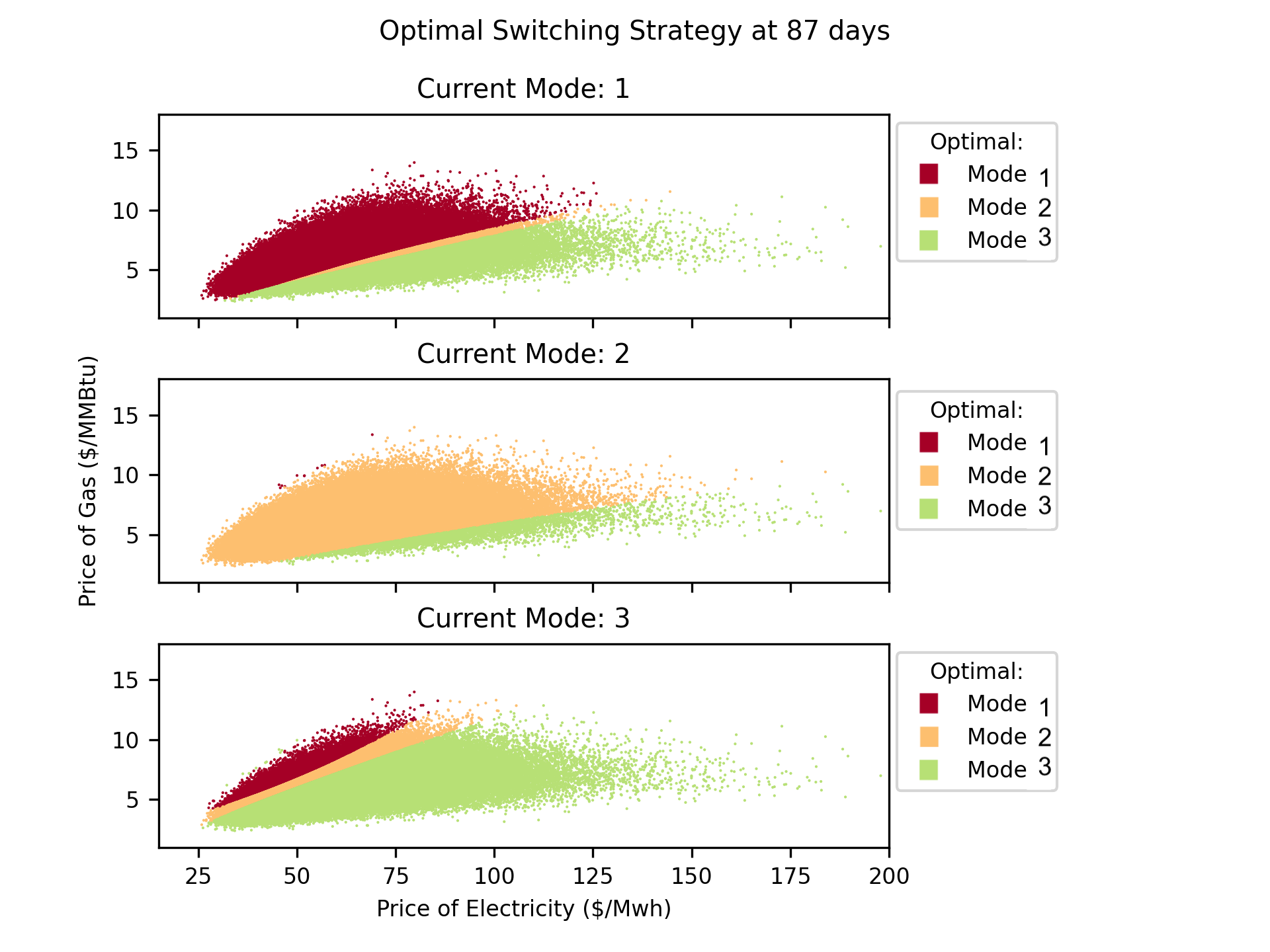}
\end{subfigure}
\begin{subfigure}{0.45\textwidth}
        \includegraphics[width = \textwidth]{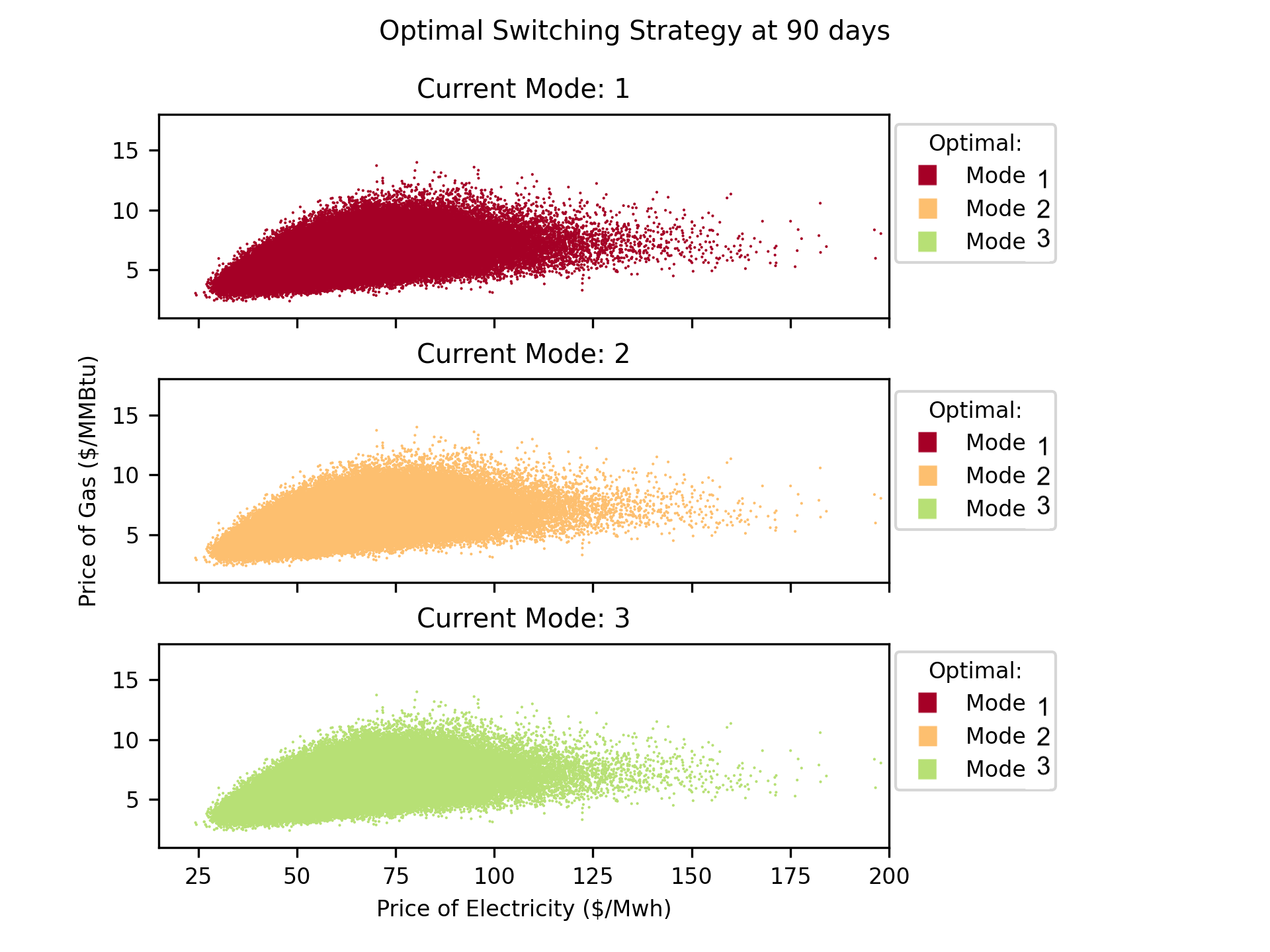}
\end{subfigure}
\caption{Switching strategies starting from mode $i = 1,2,3$ for a selection of time steps, where $N = 180$, corresponding to two electricity production decisions being made per day. The first plot of each subfigure shows the optimal mode to switch to when starting in mode 1 at the given time. The second plot corresponds to starting in mode 2 at the given time, and the third plot corresponds to starting in mode 3 at the given time. Each group of three plots corresponds to a different point in the 90-day scheduling problem. \label{fig:cl-2}}
\end{figure}

\revised{Readers may wonder about the performance for higher value of $\lambda$. If we double $\lambda$ from 8 to 16, we obtain similar error estimates, as detailed in \cref{tab:CL_results_2}, and runtimes are unaffected.}
\begin{table}
\centering
\begin{tabular}{|c|c|c|c|}
\hline
$x_0 = [50,6]$ & OSJ & LS & \revised{Difference} \\ 
\hline
$V(0, x_0, 1)$ & 1.14020& 1.14641 & 0.5\%  \\
\hline
$V(0, x_0, 2)$ & 1.19858 & 1.20930 & 0.8\% \\
\hline
$V(0, x_0, 3)$ & 1.13858 & 1.13399 & 0.4\% \\
\hline
\end{tabular}
\caption{This extends the results of \cref{tab:CL_results} to a scenario where $\lambda = 16$ instead of $\lambda = 8$. Similarly, 100,000 paths of $\{X^\pi_n\}_{n=0}^{180}$ were generated and \revised{the networks were trained with a learning rate of 0.001 over 20 epochs, and had a runtime of 13.96 minutes.}}
\label{tab:CL_results_2}
\end{table}

\subsubsection{Higher-dimensional Performance}
We can also consider a high-dimensional example where optimization is done over the price of electricity $X^0_t = P_t$ (a jump-diffusion process as in the previous example), as well as $F$ possible fuels $(X^1_t,\ldots, X^F_t)$, leading to a problem with $d = 1 + F$. The fuel prices $X^\phi_t, \phi = 1,\ldots,F$ follow an Ornstein--Uhlenbeck-type stochastic evolution given by 
\begin{equation*}
dX^\phi_t= X^\phi_t\{2(\log(6) -\log X^\phi_t) dt + 0.4(0.8 dW^1_t + 0.6 dW^\phi_t)\}.
\end{equation*}
The modes remain the same (mode 1: shut down, mode 2: half capacity, mode 3: full capacity) and the switching costs use the average over all the input fuels. The profit functions are given by $f_i(P_t, \bar X_t)$, defined as in \cref{eq:psi}, where $\bar X_t$ is the geometric mean of the fuel prices $X^1_t,\ldots,X^F_t$ and has the same distribution as $G_t$. \revised{Therefore, analytically this problem can be simplified to the two-dimensional example, and the output can be compared to that of \cref{sec:CL_2dim}. However, the neural network does not ``know" that the problems are equivalent, and simply trains the optimal weights for a higher-dimensional input vector and a larger neural network (recall there are $d+10$ nodes in each layer of the network). So, we can use this example to examine the numerical accuracy and computational performance in higher dimensions. Our results are displayed in \cref{tab:dim_results}, where we demonstrate both consistent accuracy and a linear (rather than exponential) increase in running time in problems of up to 70 dimensions. The increase in running time can also be observed visually in \cref{fig:dim_vis}.}

\begin{table}
\centering
\revised{\begin{tabular}{|c|c|c|}
\hline
Dimension ($d$) & Running time (min) & \revised{Average Difference} \\ 
\hline
2 & 13.73 & 0.433\% \\ 
\hline
10 & 19.52 &  1.004\% \\ 
\hline  
20 & 36.26  &  1.966\% \\ 
\hline
30 & 40.05  & 1.120\% \\ 
\hline
40 & 51.02 & 1.231\% \\ 
\hline
50 & 51.89 & 1.275\% \\ 
\hline
60 & 65.82 & 1.188\% \\  
\hline
70 & 71.95 & 1.617\% \\  
\hline
\end{tabular}}
\caption{This table contains the running times and errors for a series of high-dimensional examples \revised{which are analytically equivalent to the two-dimensional example solvable by the Longstaff-Schwartz algorithm. By average difference, we refer to the average over the three modes of the differences between the value functions produced by the OSJ and LS algorithms.}}
\label{tab:dim_results}
\end{table}

\begin{figure}
\centering
\includegraphics[width = 0.7\textwidth]{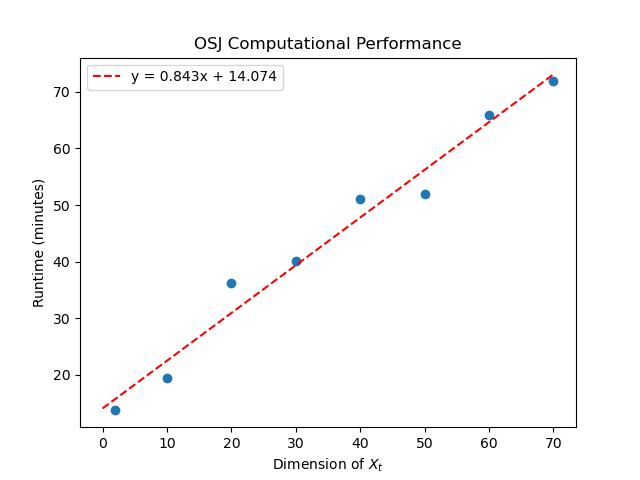}
\caption{A linear regression of the performance of the algorithm as the dimension of the problem increases. \label{fig:dim_vis}}
\end{figure}

\subsection{Optimal Capacity Decisions (Aid, Campi, Langrene, Pham \cite{Aidetal14})}
\label{sec:ex_Aid}
We now consider a high-dimensional example adapted from the paper \cite{Aidetal14}. In their paper, they investigated the question of when and how many power plants of various types to build over a forty-year time horizon. Specifically, they considered an agent who is able to build power plants which rely on a cheaper fuel or a more expensive one. The optimal construction schedule was able to decrease electricity prices while also maximizing profits. We seek to answer a related question on a shorter time horizon, where the investor owns multiple preexisting power generation facilities that produce electricity using different fuels and sources. The operator must determine the optimal configuration of these facilities to bring online at any given time in order to optimize the profits made by the operator of the facilities. However, electricity prices are also stochastic, and changes in plant operations incur switching costs. These changes are made because we would like to investigate strategies on shorter time horizons with more frequent operating decisions, where bringing plants on- and offline in response to demand is not a trivial process. Additionally, the neural network structure of our algorithm will allow us to consider higher-dimensional problems, and so we take advantage of that to model stochastic electricity prices as electricity producers do not always have full control over electricity prices. 

One feature of this example is a stochastic ``realized availability rate" for each facility, which can be explained as an indicator of fluctuations in realized electricity production due to exogenous factors affecting plant efficiency, despite each facility having a constant ``official" output capacity.
Additionally, a carbon dioxide emission penalty is included in the price of each input fuel, where the price of carbon dioxide is also stochastic. 
If we consider $F$ fuel sources (for example, natural gas and oil), the state variables are their prices $(S^\phi_t), \phi=1,\ldots,F$, their stochastic availability $(A^\phi_t), \phi=1,\ldots,F$, the price of carbon dioxide ($S^0_t$), the demand for electricity $(D_t)$, and the price of electricity $(P_t)$. 
Therefore, the total dimension of the state variable is $d = 2F + 3$ where $F$ is the number of different fuels (or more generally, energy sources) being considered. 
The ``modes" we consider are the output levels for each of these fuel sources at the given decision period. We will now proceed to describe the evolution of each of these components of the state variable.

Electricity demand $D_t$ is represented by an Ornstein--Uhlenback process $Z^0_T$ that is shifted using a cosine function $\mathcal{H}$ to incorporate seasonal demand fluctuations, so that
\begin{align*}
&dZ^0_t = \alpha^0 Z^0_t dt + \beta^0 dW^0_t, \\
&D_t = Z^0_t + \mathcal{H}(t).
\end{align*}

Realized availability ($A^\phi_t,\ \phi = 1, \ldots, F$) also fluctuates according to an Ornstein--Uhlenback process, and is transformed by a quantile function $\mathcal{T}:\R \to [0,1]$ to a percentage of the ``official" capacity of each power generation facility, as described by
\begin{align*}
&A^\phi_t = \mathcal{T}(Z^\phi_t), \\
&dZ^\phi_t = \alpha^\phi Z^\phi_t dt + \beta^ \phi dW^\phi_t.
\end{align*}

The remaining inputs to the optimal switching problem are the costs of the fuels and the electricity spot prices. In reality, spot prices are set through electricity markets involving complicated interactions between multiple players, but here we model them as stochastic and independent of the operator's switching decision. \revised{Fuel costs and electricity prices are correlated through the cointegration matrix $\mu$ and covariance matrix $\Sigma$, where $\mu$ is a matrix with rank $r$ such that $1 < r < F + 2$.} The raw fuel prices $S^\phi_t,\ \phi = 0, 1,\ldots, F$ and electricity spot prices $P_t$ are jointly referred to as $S_t = (S^0_t, ..., S^F_t, P_t) \in [0, \infty)^{F+2}$, the dynamics of which are given by
\begin{align*}
dS_t = \mu S_t dt + \text{diag}(S_t) \left(\Sigma dW^S_t + \int_{\R^F} S_t(\exp(e) - 1) \NN(de, dt)\right).
\end{align*}
This price process can also be simulated exactly. The electricity price process (the first dimension of the vector) experience jumps of size $S_t(\exp(e) - 1)$ where $\NN(de, dt)$ is a Poisson random measure with intensity $\lambda$ and $e$ follows an exponential distribution.
The total cost of each fuel is the raw cost of the fuel times its heat rate, $h_\phi S^\phi_t$, plus a carbon dioxide emission charge $h_\phi^0 S^0_t$, and we denote this price as $\tilde S^\phi_t = h^0_\phi S^0_t + h_\phi S^\phi_t$. The specific values of these parameters are given in \cref{tab:aid_params}.
All these factors combine with the current available capacities for each fuel type, given by $K^\phi_t = A^\phi_t \times M^{\phi,i}_t$, where $M^{\phi,i}_t$ is the operating level for the plant that uses fuel $\phi$ in mode $i$. 
\revised{We denote the total production capacity at time $t$ by $\bar K_t = \sum_{\phi=1}^F K^\phi_t$.}

We consider a time horizon $T = 0.25$ years (3 months), where production decisions can be made once per day ($N = 90$). 
The electricity spot prices, along with the cost of altering the capacity of each type of power generation facility and the current demand, determine the total profit made by the owner of the facilities.
The associated running cost is defined as total revenue minus the costs of operating the plants for each technology at the chosen capacities. This can be represented as
\begin{equation*}
\revised{f(t, X_t, K_t) = P_t \min\{D_t, \bar K_t\} + 0.5P_t(\bar K_t - D_t)^+ - 2P_t(D_t - \bar K_t)^+ - \sum_{\phi=1}^F K^\phi_t \tilde S^\phi_t},
\end{equation*}
where $X_t = (D_t, A^1_t,\ldots, A^F_t, S^0_t,\tilde S^1_t,\ldots, \tilde S^F_t)$ and $K_t = (K^1_t,\ldots,K^F_t)$. The basic intuitions behind this profit function are
\begin{itemize}
\item \revised{Electrical demand must be met.}
\item \revised{The power plant can sell electricity in excess of demand at a steep discount, and can buy additional electricity to satisfy demand at a steep premium.}
\item \revised{The total operational capacity is constant across all modes, but the cost and flexibility of production differs depending on the configuration of the three plants.}
\item Once mode $i$ is chosen, the power plant must make $K^\phi_t = A^\phi_t \times M^{\phi,i}_t$ units of electricity using technology $\phi$.
\end{itemize}

To demonstrate the expanded capabilities of our algorithm, we consider three different types of power plants: natural gas ($\phi=1$), coal ($\phi=2$), and nuclear ($\phi=3$) plants. \revised{Switching costs are given by
\begin{equation*}
C_{i,j}(X_t) = \sum_{\phi=1}^F \bm{c}^\phi S^\phi_t \mathbf{1}_{\{M^{\phi,i}_t \neq M^{\phi,j}_t\}}+ \epsilon,
\end{equation*}
where $\bm{c}^\phi$ imposes a scale factor on the cost of fuel $\phi$, and there is an additional small fixed switching cost $\epsilon$ to satisfy \cref{assm_cost}.
The natural gas-based electricity production facility can be ``scaled up" slightly, at the expense of efficiency (conveyed through heat rate), and the ``switching cost" of this change is determined accordingly. 
The coal-based and nuclear production facilities are supplementary facilities that can be turned on or off. 
Nuclear energy is an interesting addition to the problem described in \cite{Aidetal14} as it has lower and more consistent running costs than fossil fuels. In addition, nuclear energy production is not subject to a carbon dioxide emissions charge and the cost of uranium is assumed to be only very weakly correlated with fossil fuel pricing. However, the nuclear plant faces high friction when starting up or shutting down the nuclear reactor, as this is a slow and complicated process if done safely. Therefore, we can model this type of plant and its associated security concerns with a higher ``switching" cost, quantifying the time and energy required to turn the plant on and off. 
Finally, the coal plant also takes time and fuel to turn on and off, which is reflected by ``switching" cost associated with bringing the plant online or taking it offline. 
The possible electricity production configurations (modes) that we consider and their switching costs are listed in \cref{tab:aid}. Notice that capacity levels are the same for all modes, but the mix of energy sources varies from mode to mode. This allows us to examine the optimal electricity production configuration based on our parameter choices, listed in \cref{tab:aid_params}.}

\begin{table}[H]
\centering
\begin{tabular}{ |c|c|c|c|c|c| } 
 \hline
Fuel & $M^{\phi,1}$ & $M^{\phi,2}$ & $M^{\phi,3}$ & $M^{\phi,4}$ & Switching cost \\
 \hline
Natural Gas ($\phi = 1$) & 50 & 60 & 60 & 70 & $0.1 S^1_t$\\ 
Coal ($\phi = 2$) & 10 & 0 & 10 & 0 & $0.1 S^2_t$\\ 
Nuclear ($\phi = 3$) & 10 & 10 & 0 & 0 & $0.5 S^3_t$ \\
 \hline
\end{tabular}
\caption{The modes considered in this example involve natural gas ($\phi=1$), coal ($\phi=2$), and nuclear ($\phi=3$) energy sources. $M^{\phi,i}$ is the baseline production capacity of a plant which uses fuel $\phi$, when operating in mode $i$ for $i = 1,\ldots, 4$.}
\label{tab:aid}
\end{table}

The strategies associated with this model are investigated in Figures \ref{fig:aid1} and \ref{fig:aid2}, which display the relationships between fuel prices and optimal strategies. In each heatmap, fuel prices for one of the three fuels are held constant at their average level, and prices and strategies for the remaining two fuels are visualized. The heatmaps are presented in groups of four to illustrate how the optimal switching strategy is dependent on the current mode at time $t_n$, and each subfigure shows the optimal strategies (given by the colors in the heatmap) for different values of the state variables at a given time and current mode. This allows us to isolate how prices of natural gas, coal, and uranium individually affect optimal strategies and also to examine how switching costs contribute to switching decisions. 

\revised{Overall, our parameters defined coal as the least efficient of the three sources, and is also the most expensive once the carbon dioxide emissions charge is incorporated into the pricing. 
Nuclear energy is the most cost-effective on average, but natural gas is the most efficient when comparing heat rates.
Therefore, modes 1 and 3, which utilize the coal plant, are generally less preferred than the others.
Additionally, the ``cost" of turning on and off the nuclear reactor is quite high, as this cost encompasses the safety considerations associated operating a nuclear reactor.
However, the beneficial qualities of nuclear power are seen to outweigh the drawbacks of the increased switching cost when compared to the lower efficiency of a coal plant.}

\revised{We first discuss \ref{fig:aid1}, which displays the optimal switching strategies for a given current mode at day eighty-five of the ninety-day optimization period.
Note that in general, it is not profitable to incur a loss in revenue by changing the operating mode of the plant, so most of the region remains the color which corresponds to the current mode.
However, if switches do occur, they favor mode 2, where electricity is produced by a combination of nuclear and gas, and occur in regions where either the cost of nuclear is very low or the cost of the other fuels is very high. 
This occurrence is in line with the formulation of the problem, where nuclear and gas power perform better than coal power. 
Empirically, therefore, the switching strategies provided by the algorithm appear reasonable.}

\revised{Furthermore, as $n$ decreases (meaning that there is more time left for the operator to make switching decisions before the end of the operational period at $n=N=90$), switching costs make up a smaller proportion of total costs. 
We look at \ref{fig:aid2}, which shows the optimal switching strategies at day thirty -- much earlier in the process. 
At this point, there is also more uncertainty regarding the evolution of the state process, and this appears to correlate with the algorithm favoring the mode with the best ``average" performance, which is mode 2, where electricity is produced by a combination of natural gas and nuclear power.
Overwhelmingly, the region for which it is optimal to switch to mode 2 increases as $n$ decreases, as shown in \cref{fig:aid2}.}

Note that in \cref{fig:aid2} the algorithm does not prescribe a switch when gas prices are high and when operating fully using gas. This may seem counterintuitive, as higher gas prices in comparison to other fuel prices make gas less attractive at first glance. However, these higher gas prices also result in much higher switching costs for the plant, as the switching costs are also directly proportional to the fuel costs in our model. This introduces an element of reluctance if the profit to be made from switching is small enough in comparison to expensive switching costs. In addition, the mean reversion of the gas prices implies downward pressure on high gas prices (in expectation). It can be seen that as the time horizon becomes longer this effect is weaker, but it still has some impact in \cref{fig:aid2}.

\begin{figure}[H]
\centering
\includegraphics[width = \textwidth]{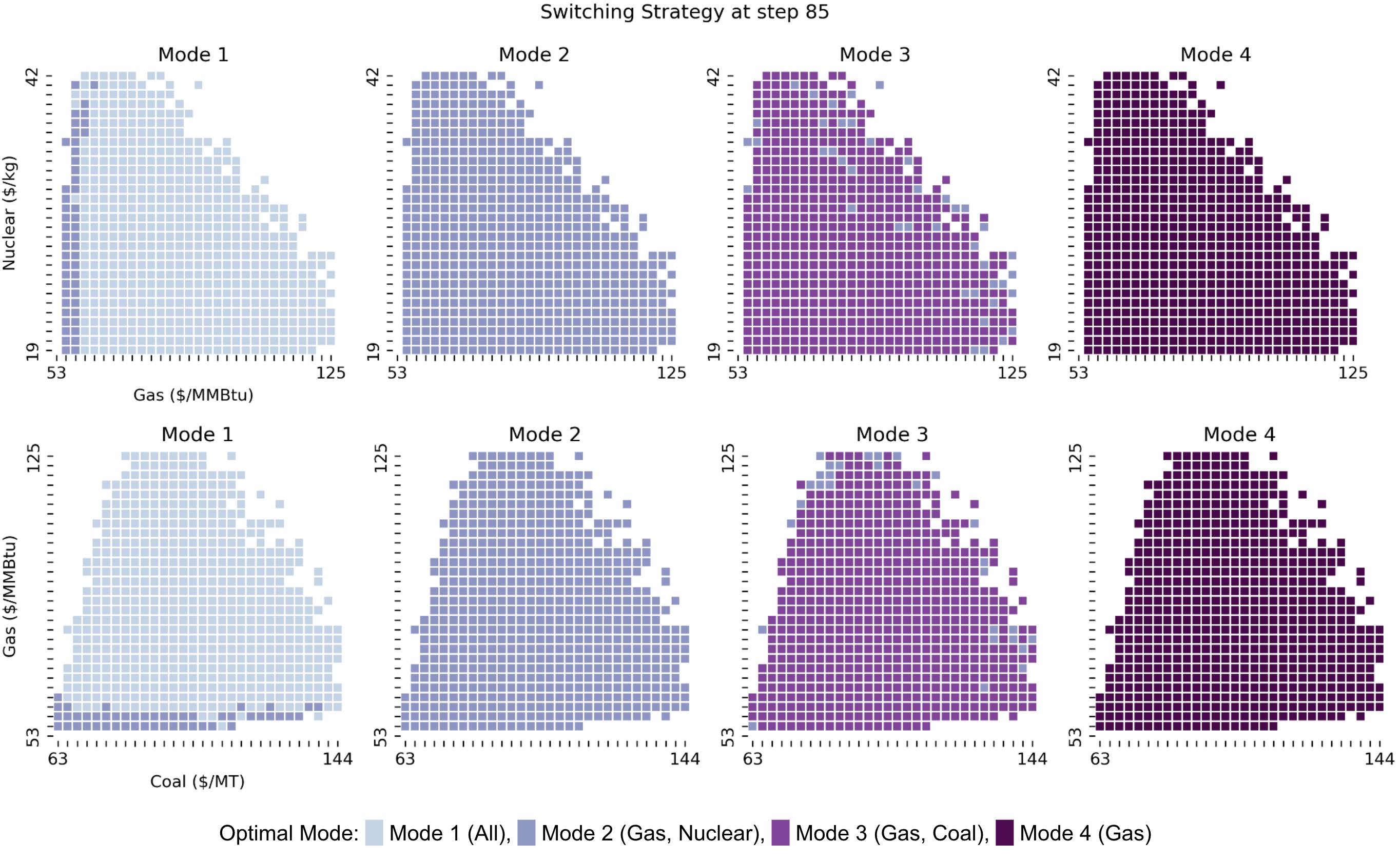}
\caption{Switching strategies at $n=85$, where $n$ is the number of days that has elapsed since the start of the period and the period has length $N = 90$ days. 
 \label{fig:aid1}}
\end{figure}

\begin{figure}[H]
\centering
\includegraphics[width = \textwidth]{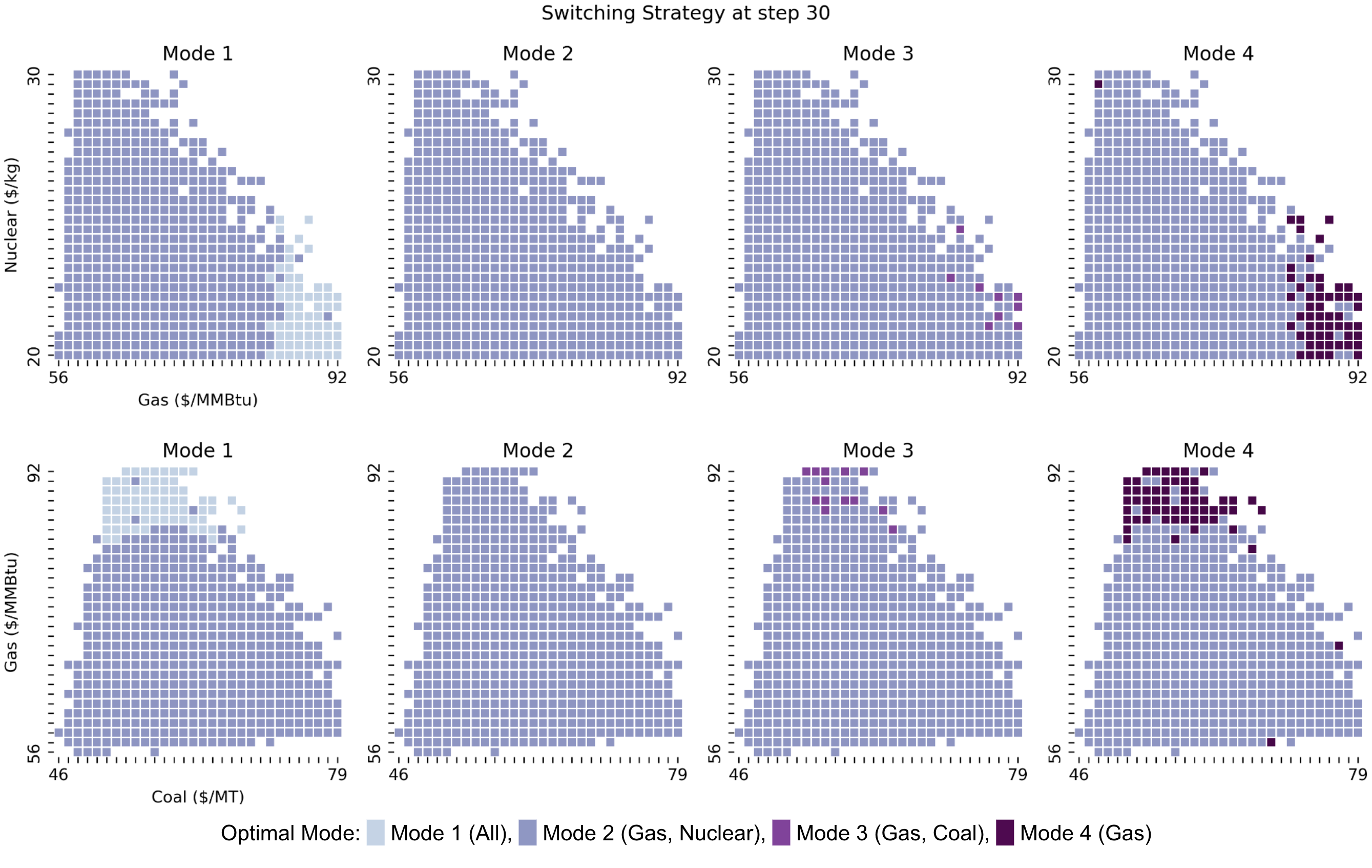}
\caption{Switching strategies at $n=30$, where $n$ is the number of days that has elapsed since the start of the period and the period has length $N = 90$ days. 
 \label{fig:aid2}}
\end{figure}


\section{Convergence Analysis}
\label{sec:proof}
Recall that we wish to prove that the error of the OSJ algorithm will converge to zero as the neural network approximation errors of the architecture converge to zero and as the size of the time steps, $\Delta t = T/M$, decreases to zero. The error in both \cref{thm:main} and \cref{thm:strat} can be split into two parts: a discretization error for the stochastic processes and an algorithm-associated error. Once we have established the discretization error, we will be able to calculate the error of the neural network approximation by working with only discrete processes.

\subsection{Discretization Error}
\label{sec:discretization}
As stated in \cref{rmk:disc}, the result of \cref{thm:main} follows when the discretization scheme used for simulating paths of $X_t$ is of at least strong order 0.5 and the result of \cref{thm:strat} follows when the discretization scheme is of at least strong order 1.0 (the choice of scheme is arbitrary). For consistency, we explicitly define the discretization error as
\begin{equation*}
\left(\E\left[\max_{n = 1, \ldots, N-1} |X_{t_{n+1}} - X^\pi_{n+1}|^2 + \sup_{t \in [t_n, t_{n+1}]} |X_t - X^\pi_n|^2\right]\right)^{1/2}\leq C_X (\Delta t)^\gamma = O(M^{-\gamma})
\end{equation*}
when the discretization scheme is strong order $\gamma$. For our purposes, $\gamma = 0.5$ in \cref{sec:y-proof} and \cref{sec:zu-proof} and $\gamma = 1.0$ in \cref{sec:strat-proof}.

Let us now move on to the error between $Y^i_t = V(t, X_t, i)$ and the discrete approximation of this process, $Y^{\pi,i}_n$, which approximates $Y^i_t$ at time $t_n$ starting in mode $i$. At the terminal time $T$, the terminal condition dictates that
\begin{equation*}
Y^{\pi,i}_M \equiv g^i(X^\pi_M).
\end{equation*}
Earlier values of $Y^{\pi,i}_n$ for $n \in \{0, \ldots, M-1\}$ can now be defined recursively in terms of the discrete approximation of the continuation values at each time step, represented by $\tilde Y^{\pi,i}_n$, following the relationship
 \begin{equation}
\label{eq:switch}
\revised{Y^{\pi,i}_n := \mathbf{1}_{t_n \in \mathfrak{R}} \max\left\{\tilde Y^{\pi,i}_n,\ \max_{j \neq i}\big(- C_{i,j}(X^\pi_n) + \tilde Y^{\pi,j}_n\big)\right\}+ \mathbf{1}_{t_n \notin \mathfrak{R}}\tilde Y^{\pi,i}_n, \forall i \in \I, \forall n \in \{0,\ldots, M-1\}. }
\end{equation}
Recall that $\mathfrak{R}$ is the switching grid, and the grid size is of order $O(M^{-1/2})$.
This representation is driven by the dynamic programming principle.
Therefore, we are able to solve for $Y^{\pi,i}_n$ by comparing the continuations values $\tilde Y^{\pi,i}_n$ at each step $n$. 
In turn, the continuation values at time $t_n$ rely on the approximated value functions $Y^{\pi,i}_{n+1}$ from the previous step. 
The exact formulation of the discrete approximation of these continuation values, as well as discrete approximations of $Z^i_t$ and $\Delta Y^i_t$, follows the method laid out in \cite{BouchardElieApprox}. We set
\begin{align}
\tilde Y^{\pi,i}_n  & := \E[Y^{\pi,i}_{n+1}|\F_n] +  f(t_n, X^\pi_n)\Delta t, \label{eq:disc_y}\\
Z^{\pi, i}_n &:= \frac{1}{\Delta t} \E[Y^{\pi,i}_{n+1} \Delta W_n|\F_n], \label{eq:disc_z}\\
\Delta Y^{\pi,i}_n &:=  \frac{1}{\lambda \Delta t} \E[Y^{\pi,i}_{n+1}\Delta \tilde N_n|\F_n]. \label{eq:disc_u}
\end{align}
\revised{The strong order $\gamma$ approximation of $X_t$ where $\gamma \geq 0.5$ and the conditions in \cref{assm_beta_etc} ensure that the discrete-time approximation scheme described in \cref{eq:disc_y}-\cref{eq:disc_u} is of strong order at least 0.5, as stated in Theorem 2.1 of \cite{BouchardElieApprox}, when considering the un-reflected BSDEs with jumps. Theorem 5.4 of the reference \cite{ChasElieKhar12} presents the convergence rate for a similar discretization scheme for a multidimensional reflected BSDE without jumps. By following a similar path to these two papers, one which handles a BSDE with jumps and one which handles a BSDE with reflections which arise from switching, we can obtain a convergence rate for the above discretization scheme. For all modes $i$ in $\I$ and for $\e$ arbitrarily small, we have
\begin{equation}
\label{eq:disc-approx-err}
\begin{split}
\max_{n \in 0,1,\ldots,M} \E[\tilde Y^i_{t_n} - \tilde Y^{\pi,i}_n|^2 + |Y^i_{t_n} - Y^{\pi,i}_n|^2] &\leq C^\e M^{-1 + \e}, \\
\E\left[\sum_{n = 0}^{M-1} \int_{t_n}^{t_{n+1}} |Z^i_t - Z^{\pi,i}_n|^2 dt \right] &\leq C^\e M^{-1/2 + \e},\\
\E\left[\sum_{n = 0}^{M-1} \int_{t_n}^{t_{n+1}} \left|\int_{\R^d} \Delta Y^i_t(e) \nu(de) - \Delta Y^{\pi,i}_n\right|^2 ds \right] & \leq C^\e M^{-1/2 + \e}.
\end{split}
\end{equation}}
Therefore, in following sections, we will focus on the error between the discrete approximations $(Y^{\pi,i}_n, Z^{\pi,i}_n, \Delta Y^{\pi,i}_n)$ and the neural network approximations $(\hat \YY^i_n(X^\pi_n),\hat \ZZ^i_n(X^\pi_n),\widehat{\Delta Y}^i_n(X^\pi_n))$.

\begin{remark}
Each $\tilde Y^{\pi,i}_n$ can be represented as a conditional expectation with respect to $\F_n$, so due to the Markovian nature of the problem, it can be expressed as a function of $X^{\pi,i}_n$. We abuse notation and denote the value function by $\tilde Y^{\pi,i}_n(X^{\pi,i}_n)$. The same can be done to represent $Y^{\pi,i}_n$ as $Y^{\pi,i}_n(X^\pi_n)$, $Z^{\pi,i}_n$ as $Z^{\pi,i}_n(X^\pi_n)$, and $\Delta Y^{\pi,i}_n$ as $\Delta Y^{\pi,i}_n(X^\pi_n)$. 
\end{remark}

\subsection{Proof of \cref{thm:main}}
\subsubsection{Value Function Learning Error}
\label{sec:y-proof}
We now want to bound the error between the function $\hat \YY^i_n(X^\pi_n)$ learned by the OSJ algorithm and the discrete value function approximation $Y^{\pi,i}_n(X^\pi_n)$, for any given starting mode $i$. Using $|\max\{a,b\} - \max\{c,d\}| \leq \max\{|a - c|, |b - d|\}$ we can rewrite 
\begin{align*}
|Y^{\pi,i}_n(X^\pi_n) - \hat \YY^i_n(X^\pi_n)| =& |\max\{\tilde Y^{\pi,i}_n(X^\pi_n), \max_{j \neq i}(-C_{i,j}(X^\pi_n) + \tilde Y_n^{\pi,j}(X^\pi_n))\} \\
&- \max\{\tilde \YY^i_n(X^\pi_n), \max_{j\neq i}(-C_{i,j}(X^\pi_n)+\tilde \YY^j_n(X^\pi_n))\}|\\
\leq& \max\{|\tilde Y^{\pi,i}_n(X^\pi_n) - \tilde \YY^i_n(X^\pi_n)|, \\
&\hspace*{1cm}|\max_{j \neq i}(-C_{i,j}(X^\pi_n) + \tilde Y_n^{\pi,j}(X^\pi_n) + C_{i,j}(X^\pi_n) - \tilde \YY^j_n(X^\pi_n))|\} \\
=&  \max\{|\tilde Y^{\pi,i}_n(X^\pi_n) - \tilde \YY^i_n(X^\pi_n)|, |\max_{j \neq i} (\tilde Y_n^{\pi,j}(X^\pi_n) - \tilde \YY^j_n(X^\pi_n))|\} \\
\leq &\max_{j \in \I} |\tilde Y^{\pi,j}_n(X^\pi_n) - \tilde \YY^j_n(X^\pi_n)|.
\end{align*}
Therefore, we conclude that
\begin{equation}
\label{eq:key}
\E\big[ \max_{i \in \I} | Y^{\pi,i}_n(X^\pi_n) - \hat \YY^i_n(X^\pi_n)|^2 \big]\leq \E\big[ \max_{i \in \I} |\tilde Y^{\pi,i}_n(X^\pi_n) - \tilde \YY^i_n(X^\pi_n)|^2 \big].
\end{equation}

In light of this, let us focus on the mode-wise maximum of the difference in the continuation values $\tilde \YY^i_n$ and $\tilde Y^{\pi,i}_n$ by using \cref{eq:hat_y} and \cref{eq:disc_y} to find that
\begin{equation*}
\tilde Y^{\pi,i}_n(X^\pi_n) - \hat y^i_n(X^\pi_n) =\E\big[Y^{\pi, i}_{n+1}(X^\pi_{n+1}) - \hat \YY^i_{n+1}(X^\pi_{n+1}) | \F_n\big]. 
\end{equation*}
Squaring and taking expectation over the mode-wise maximum, then applying Jensen's inequality and tower property,
\begin{equation}
\label{eq:upper-y}
\begin{split}
\E\big[\max_{i \in \I} |\tilde Y^{\pi,i}_n(X^\pi_n) - \hat y^i_n(X^\pi_n)|^2\big] \leq \E\big[\max_{i \in \I}|Y^{\pi, i}_{n+1}(X^\pi_{n+1}) - \hat \YY^i_{n+1}(X^\pi_{n+1})|^2\big]. 
\end{split}
\end{equation}
We can get a lower bound on the left-hand side by using $\max_{i\in\I} | a_i - b_i| \geq |\max_{i\in\I} |a_i| - \max_{i\in\I} |b_i||$, such that 
\begin{equation*}
\E\left[\max_{i \in \I} \big|\tilde Y^{\pi, i}_n(X^\pi_n)- \hat y^i_n(X^\pi_n)|^2\right] \geq \E\Big[\big|\max_{i \in \I} |\tilde Y^{\pi, i}_n(X^\pi_n)- \tilde\YY^i_n(X^\pi_n)| - \max_{i \in \I} |\tilde\YY^j_n(X^\pi_n) - \hat y^j_n(X^\pi_n)|\big|^2\Big].
\end{equation*}
Now, we apply Young's inequality in the form $(a - b)^2 \geq (1 - \Delta t)a^2 - \frac{1}{\Delta t}b^2$ to the right-hand side so that a lower bound is obtained of the form
\begin{equation*}
\begin{split}
&\E\big[ \max_{i \in \I} |\tilde Y^{\pi, i}_n(X^\pi_n)- \hat y^i_n(X^\pi_n)|^2\big] \\
\geq  &(1 - \Delta t) \E\big[ \max_{i \in \I} |\tilde Y^{\pi, i}_n(X^\pi_n) - \tilde \YY^i_n(X^\pi_n)|^2 \big]- \frac{1}{\Delta t}  \E\big[\max_{i \in \I} |\tilde \YY^i_n(X^\pi_n) - \hat y^i_n(X^\pi_n)|^2\big]\\
\geq &(1 - \Delta t) \E\big[ \max_{i \in \I} |\tilde Y^{\pi, i}_n(X^\pi_n) - \tilde \YY^i_n(X^\pi_n)|^2 \big]- \frac{1}{\Delta t}  \E\Big[\sum_{i \in \I} |\tilde \YY^i_n(X^\pi_n) - \hat y^i_n(X^\pi_n)|^2\Big]
\end{split}
\end{equation*}
By combining this with \cref{eq:upper-y} we can write
\begin{equation*}
\begin{split}
(1 - \Delta t) \E\big[ \max_{i \in \I} |\tilde Y^{\pi, i}_n(X^\pi_n)- \tilde \YY^i_n(X^\pi_n)|^2 \big] \leq & \E\big[ \max_{i \in \I} |\tilde Y^{\pi, i}_n(X^\pi_n)- \hat y^i_n(X^\pi_n)|^2\big] \\
&+ \frac{1}{\Delta t} \E\Big[\sum_{i \in \I} |\tilde \YY^i_n(X^\pi_n) - \hat y^i_n(X^\pi_n)|^2\Big]  \\
\leq &\E\big[ \max_{i \in \I} |Y^{\pi, i}_{n+1}(X^\pi_{n+1}) - \hat \YY^i_{n+1}(X^\pi_{n+1})|^2\big] \\
& + \frac{1}{\Delta t}\sum_{i \in \I} \E |\tilde \YY^i_n(X^\pi_n) - \hat y^i_n(X^\pi_n)|^2.
\end{split}
\end{equation*}
For $\Delta t$ sufficiently small (less than 1), we have
\begin{equation}
\begin{split}
\label{eq:cont-bound}
\E\big[ \max_{i \in \I}|\tilde Y^{\pi, i}_n(X^\pi_n)- \tilde \YY^i_n(X^\pi_n)|^2\big] \leq& (1 + C_1\Delta t) \E\big[ \max_{i \in \I} |Y^{\pi, i}_{n+1}(X^\pi_{n+1}) - \hat \YY^i_{n+1}(X^\pi_{n+1})|^2 \big ]\\
&+  C_1M\sum_{i \in \I} \E |\tilde \YY^i_n(X^\pi_n) - \hat y^i_n(X^\pi_n)|^2,
\end{split}
\end{equation}
where $C_1$ does not grow as $M$ increases and is independent of structure of the neural networks. 

At this point, we have established an upper bound on the error between the discretized continuation values $\tilde Y^{\pi, i}_n$ and the neural-network-generated continuation values $\tilde \YY^i_n(X^\pi_n)$. 
This upper bound is in terms of $\E|\tilde \YY^i_n(X^\pi_n) - \hat y^i_n(X^\pi_n)|^2$ and $\E|Y^{\pi, i}_{n+1}(X^\pi_{n+1})-\hat \YY^i_{n+1}(X^\pi_{n+1})|^2$. 
However, \cref{thm:main} states the convergence error in terms of the neural network approximation errors $\e^y_n$, $\e^z_n,$ and $\e^u_n$.
Therefore, our next course of action is to determine a bound on $\E|\tilde \YY^i_n(X^\pi_n) - \hat y^i_n(X^\pi_n)|^2$ in terms of these neural network approximation errors.

To do this, we must analyze the loss function described in \cref{eq:loss}. Replace $\hat \YY^i_{n+1}(X^\pi_{n+1})$ in \cref{eq:loss} with its corresponding BSDE representation \cref{eq:loss-bsde} and use the relations \cref{eq:z_hat} and \cref{eq:u_hat} to write
\begin{align*}
L^i_n(\theta) = \E\Bigg|& \hat y^i_n(X^\pi_n) - \YY^i_n(X^\pi_n, \theta) + \big( f_i(t_n, X^\pi_n) -  f_i(t_n, X^\pi_n)\big) \Delta t\\
&+ \int_{t_n}^{t_{n+1}} \int_{\R^d} u^i_s(e) \tilde \NN(de, ds)  -  \Delta \YY^i_n(X^\pi_n, \theta)\Delta \tilde N_n  + \int_{t_n}^{t_{n+1}} (z^i_s)^T d W_s - \ZZ^i_n(X^\pi_n, \theta)^T\Delta W_n \Bigg|^2\\
= \E\big|& \hat y^i_n(X^\pi_n) - \YY^i_n(X^\pi_n, \theta) +\big(\hat u^i_n(X^\pi_n) -  \Delta \YY^i_n(X^\pi_n, \theta)\big)\Delta \tilde N_n  + \big(\hat z^i_n(X^\pi_n) - \ZZ^i_n(X^\pi_n, \theta)\big)^T\Delta W_n \big|^2\\
+\E&\left[\int_{t_n}^{t_{n+1}} \int_{\R^d} |u^i_s - \hat u^i_n(X^\pi_n)|^2 \NN(de, ds)\right] + \E\left[\int_{t_n}^{t_{n+1}} ||z^i_s - \hat z^i_n(X^\pi_n)||^2 ds\right]\\
=\bar L^i_n&(\theta) + Error(U, Z),
\end{align*}
where we used It\^o isometry \revised{and the definitions of $\hat u^i_n$ and $\hat z^i_n$, given in \cref{eq:hat_u} and \cref{eq:hat_z} respectively,} to split up the expectation.
We introduce the notation
\begin{equation*}
\bar L^i_n(\theta) := \E\big| \hat y^i_n(X^\pi_n) - \YY^i_n(X^\pi_n, \theta) +\big(\hat u^i_n(X^\pi_n) -  \Delta \YY^i_n(X^\pi_n, \theta)\big)\Delta \tilde N_n  + \big(\hat z^i_n(X^\pi_n) - \ZZ^i_n(X^\pi_n, \theta)\big)^T\Delta W_n \big|^2
\end{equation*} 
and 
\begin{equation*}
 Error(U, Z) := \E\left[\int_{t_n}^{t_{n+1}} \int_{\R^d} |u^i_s - \hat u^i_n(X^\pi_n)|^2 \NN(de, ds)\right] + \E\left[\int_{t_n}^{t_{n+1}} ||z^i_s - \hat z^i_n(X^\pi_n)||^2 ds\right].
\end{equation*}
Note that $Error(U,Z)$ is independent of our choice of $\theta$. In addition, this error scales with $\Delta t$ as stated in \cite{BouchardElieApprox}, since $\hat z^i_n$ and $\hat u^i_n$ are the $L^2$ projections of $z^i_s$ and $u^i_s$ on $[t_n, t_{n+1})$.
For the near future, we will focus on $\bar L^i_n(\theta)$. 
By It\^o isometry and $\E[\Delta W_n] = 0$, $\E[\Delta \tilde N_n] = 0$, we obtain 
\begin{equation*}
\begin{split}
\bar L^i_n(\theta) =& \E|\hat y^i_n(X_n) - \YY^i_n(X_n, \theta)|^2 \\
& + \Delta t \E||\hat z^i_n(X_n) - \ZZ^i_n(X_n, \theta)||^2 + \lambda \Delta t \E|\hat u^i_n(X_n) - \Delta \YY^i_n(X_n, \theta)|^2.
\end{split}
\end{equation*}
Extending this to the sum of the loss functions over all possible modes and choosing $\theta  = (\theta_1, \theta_2, \theta_3) = \theta^{*,i}_n \in \argmin_{\theta \in \Theta} L^i_n(\theta)$, we conclude that
\begin{equation}
\begin{split}
\label{eq:loss-bd1}
\sum_{i \in \I} \bar L^i_n(\theta^{*,i}_n) = & \sum_{i \in \I} \E|\hat y^i_n(X^\pi_n) - \tilde \YY^i_n(X^\pi_n)|^2 \\
&+\sum_{i \in \I}\Delta t \big(\E||\hat z^i_n(X^\pi_n) - \hat \ZZ^i_n(X^\pi_n)||^2 + \lambda \E|\hat u^i_n(X^\pi_n) - \widehat{\Delta \YY}^i_n(X^\pi_n)|^2\big)\\
 \leq& \sum_{i \in \I}\big(\inf_{\theta_1} \E|\hat y^i_n(X^\pi_n) - \YY^i_n(X^\pi_n, \theta_1)|^2 + \Delta t  \inf_{\theta_2} \E||\hat z^i_n(X^\pi_n) - \ZZ^i_n(X^\pi_n, \theta_2)||^2\big)\\ 
 &+\sum_{i \in \I}\lambda \Delta t \big(\inf_{\theta_3} \E|\hat u^i_n(X^\pi_n)  -  \Delta \YY^i_n(X^\pi_n, \theta_3)|^2\big)\\
=&\e^y_n + \Delta t(\e^z_n + \lambda \e^u_n).
\end{split}
\end{equation}
We note that $\theta^{*,i}_n$ must minimize both $L^i_n(\theta)$ and $\bar L^i_n(\theta)$ because $Error(U,Z)$ does not depend on $\theta$.  
This facilitates the above inequality, where the final line follows from the definition of the neural network errors given in \cref{eq:eps}.
From \cref{eq:loss-bd1}, we also get the looser bound
\begin{equation}
\label{eq:loss-bd2}
\sum_{i \in \I}  \E|\hat y^i_n(X^\pi_n) - \tilde \YY^i_n(X^\pi_n)|^2 \leq C_2(\e^y_n + \Delta t(\e^z_n + \e^u_n)) = C_2 \bm{\e^M_n}/M,
\end{equation}
where $C_2 = \max\{\lambda, 1\}$.
Applying this to \cref{eq:cont-bound} and recalling \cref{eq:key} yields 
\begin{equation}
\label{eq:pre-err}
\begin{split}
\E\big[ \max_{i \in \I}|Y^{\pi, i}_n(X^\pi_n)- \hat \YY^i_n(X^\pi_n)|^2\big] \leq&  (1 + C_1\Delta t) \E\big[ \max_{i \in \I} |Y^{\pi, i}_{n+1}(X^\pi_{n+1}) - \hat \YY^i_{n+1}(X^\pi_{n+1})|^2 \big ] \\
&+ C_1C_2M(\bm{\e^M_n}/M).
\end{split}
\end{equation}
This setup allows us to perform induction on the inequality, continuing until the right-hand side is expressed in terms of $\E\big[ \max_{i \in \I} |Y^{\pi, i}_M(X^\pi_M) - \hat \YY^i_M(X^\pi_M)|^2 \big ] =0$ and a sum of neural network errors at time steps from $n$ to $M$. Then, we can conclude the maximum error over the $M$ time steps is bounded as
\begin{equation}
\label{eq:disc-y-err}
\max_{n=0,1,\ldots,M-1}  \E\Big[ \max_{i \in \I} |Y^{\pi,i}_n(X^\pi_n) - \hat \YY^i_n(X^\pi_n)|^2\Big]\leq C_3\sum_{n=0}^{M-1} \bm{\e^M_n}. 
\end{equation}
where $C_3$ encompasses the coefficients $C_1$ and $C_2$, as well as coefficients arising from the inductive step. This aggregate coefficient, $C_3$, is independent of $\Delta t$ and the neural network structure. 
Recall we initialized $\hat \YY^i_M(\cdot) = g^i(\cdot)$, so $ \E[\max_{i \in \I}|Y^{\pi,i}_M(X^\pi_M) - \hat \YY^i_M(X^\pi_M)|^2] = 0$. However, the error in \cref{eq:disc-y-err} is only the error between the discrete approximation of the value function and the trained neural network approximation of the value function. The true error involves the continuous value function $Y^i_t$ when starting in mode $i$ at time $t$, and so the discrete approximation error given by \cref{eq:disc-approx-err} must be incorporated into our final error bound. Therefore, the overall value function error is given by
\begin{equation}
\label{eq:y-err}
\max_{n=0,1,\ldots,M-1} \E\Big[ \max_{i \in \I}|Y^i_{t_n} - \hat \YY^i_n(X^\pi_n)|^2\Big] \leq C_3\sum_{n=0}^{M-1} \bm{\e^M_n} + \revised{C^\e M^{-1/2 + \e}}.
\end{equation}


\subsubsection{Auxiliary Function Learning Errors} 
\label{sec:zu-proof}
We must now verify that $Z^{\pi,i}_n$ and $\Delta Y^{\pi,i}_n$ are also approximated well. Looking first at the error in $Z$, we use triangle inequality to split each element of the sums into two parts and use \cref{eq:loss-bd1} on the summation to get
\begin{align*}
\Delta t \max_{i \in \I} \E ||Z^{\pi, i}_n - \hat \ZZ^i_n (X^\pi_n)||^2 \leq & 2\Delta t(\max_{i \in \I} \E||Z^{\pi, i}_n - \hat z^i_n (X^\pi_n)||^2\big] + \max_{i \in \I} \E||\hat z^i_n(X^\pi_n) - \hat \ZZ^i_n (X^\pi_n)||^2 )\\
\leq & 2\Delta t\left(\max_{i \in \I} \E||Z^{\pi, i}_n - \hat z^i_n (X^\pi_n)||^2\big] + \sum_{i \in \I} \E||\hat z^i_n(X^\pi_n) - \hat \ZZ^i_n (X^\pi_n)||^2 \right)\\
\leq & 2 \Delta t \max_{i \in \I}  \E||Z^{\pi, i}_n - \hat z^i_n (X^\pi_n)||^2 + C_2\big(\e^y_n + \Delta t\e^z_n + \Delta t \e^u_n \big).
\end{align*}
Similarly for $\Delta Y$, we find that
\begin{align*}
\Delta t \max_{i \in \I}  \E |\Delta Y^{\pi, i}_n - \widehat{\Delta \YY}^i_n (X^\pi_n)|^2 \leq & 2\Delta t( \max_{i \in \I}\E |\Delta Y^{\pi, i}_n - \hat u^i_n (X^\pi_n)|^2 + \max_{i \in \I}  \E|\hat u^i_n(X^\pi_n) - \widehat{\Delta \YY}^i_n (X^\pi_n)|^2 )\\
\leq & 2\Delta t\left( \max_{i \in \I}\E |\Delta Y^{\pi, i}_n - \hat u^i_n (X^\pi_n)|^2 + \sum_{i \in \I}  \E|\hat u^i_n(X^\pi_n) - \widehat{\Delta \YY}^i_n (X^\pi_n)|^2 \right)\\
\leq & 2 \Delta t \max_{i \in \I}  \E|\Delta Y^{\pi, i}_n - \hat u^i_n (X^\pi_n)|^2 + C_2\big(\e^y_n + \Delta t\e^z_n + \Delta t \e^u_n \big).
\end{align*}
We now need to work with $\Delta t  \E||Z^{\pi, i}_n - \hat z^i_n (X^\pi_n)||^2$ and $\Delta t \E|\Delta Y^{\pi, i}_n - \hat u^i_n (X^\pi_n)|^2$.
From the definitions of $\Delta Y^{\pi, i}_n$ and $\hat u^i_n$ and Cauchy--Schwartz of form $|\E[XY|A]|^2 \leq \E[X^2|A]\E[Y^2|A]$, we get
\begin{align*}
|\Delta Y^{\pi, i}_n - \hat u^i_n (X^\pi_n)|^2 =& \frac{1}{(\lambda \Delta t)^2} \big|\E[(Y^{\pi, i}_{n+1}(X^\pi_{n+1}) - \hat \YY^i_{n+1}(X^\pi_{n+1}))\Delta \tilde N_n |\F_n]\big|^2\\
\leq & \frac{1}{(\lambda \Delta t)^2} \E[ |Y^{\pi, i}_{n+1}(X^\pi_{n+1}) - \hat \YY^i_{n+1}(X^\pi_{n+1})|^2\big|\F_n] \E[(\Delta \tilde N_n)^2 |\F_n]\\
=& \frac{1}{\lambda \Delta t} \E[ |Y^{\pi, i}_{n+1}(X^\pi_{n+1}) - \hat \YY^i_{n+1}(X^\pi_{n+1})|^2\big|\F_n].
\end{align*}
Notice that again we used It\^o isometry applied to $\E[(\Delta \tilde N_n)^2 |\F_n]$. Therefore, it can be seen that
\begin{align*}
 \Delta t \max_{i \in \I}\E|\Delta Y^{\pi, i}_{n}(X^\pi_n) -  \hat u^i_n(X^\pi_n) |^2 \leq&  \frac{1}{\lambda}\max_{i \in \I} \E\Big[\E[|Y^{\pi, i}_{n+1}(X^\pi_{n+1}) - \hat \YY^i_{n+1}(X^\pi_{n+1})|^2 |\F_n]\Big] \\
=& \frac{1}{\lambda}\max_{i \in \I} \E|Y^{\pi, i}_{n+1}(X^\pi_{n+1}) - \hat \YY^i_{n+1}(X^\pi_{n+1})|^2\\
\leq& \frac{1}{\lambda} \E[\max_{i \in \I} |Y^{\pi, i}_{n+1}(X^\pi_{n+1}) - \hat \YY^i_{n+1}(X^\pi_{n+1})|^2].
\end{align*}
Using a similar methodology and recalling that $\Delta W_n$ is $d$-dimensional and so $\E[||\Delta W_n||^2|\F_n] = d \Delta t$, we can also derive
\begin{equation*}
\Delta t\max_{i \in \I}\E ||Z^{\pi, i}_{n}(X^\pi_n) -  \hat z^i_n(X^\pi_n) ||^2 \leq  d\E\big[ \max_{i \in \I}\big|Y^{\pi, i}_{n+1}(X^\pi_{n+1}) - \hat \YY^i_{n+1}(X^\pi_{n+1})\big|^2\big].
\end{equation*}
Let $\kappa = d +  \frac{1}{\lambda}$, so we can write
\begin{equation*}
\begin{split}
&\Delta t  \max_{i \in \I}\big( \E ||Z^{\pi, i}_n - \hat z^i_n (X^\pi_n)||^2 + \E |\Delta Y^{\pi, i}_n - \hat u^i_n (X^\pi_n)|^2\big) \\
\leq &\kappa \E\big[\max_{i \in \I}|Y^{\pi, i}_{n+1}(X^\pi_{n+1}) - \hat \YY^i_{n+1}(X^\pi_{n+1})|^2\big].
\end{split}
\end{equation*}
Now sum over all $M$. We use $Y^{\pi,i}_M(x) \equiv \hat \YY^i_M(x) \equiv g^i(x)$ and \revised{induction on \cref{eq:pre-err} (similar to the calculations done to yield \cref{eq:disc-y-err})} to yield
\begin{align*}
	&\sum_{n = 0}^{M-1} \Delta t \max_{i \in \I} \big(\E||Z^{\pi, i}_n(X^\pi_n) - \hat z^i_n(X^\pi_n)||^2 + \E |\Delta Y^{\pi, i}_n - \hat u^i_n (X^\pi_n)|^2 \big) \\
	\leq& \sum_{n = 0}^{M-1} \kappa\E\big[\max_{i \in \I}|Y^{\pi, i}_{n+1}(X^\pi_{n+1}) - \hat \YY^i_{n+1}(X^\pi_{n+1})|^2\big] \\
	=& \sum_{n = 1}^{M-1} \kappa\E\big[\max_{i \in \I}|Y^{\pi, i}_n(X^\pi_n) - \hat \YY^i_n(X^\pi_n)|^2\big] +\kappa\E\big[\max_{i \in \I}|Y^{\pi, i}_M(X^\pi_M) - \hat \YY^i_M(X^\pi_M)|^2\big] \\
	\leq & C_3\kappa\sum_{n=0}^{M-1}\bm{\e^M_n}.
\end{align*}
Therefore, 
\begin{align*}
\sum_{n = 0}^{M-1}\Delta t \max_{i \in \I} \E ||Z^{\pi, i}_n - \hat \ZZ^i_n (X^\pi_n)||^2 \leq& C_4\sum_{n=0}^{M-1}\bm{\e^M_n}, \\
\sum_{n = 0}^{M-1}\Delta t \max_{i \in \I}  \E |\Delta Y^{\pi, i}_n - \widehat{\Delta \YY}^i_n (X^\pi_n)|^2 \leq& C_4\sum_{n=0}^{M-1}\bm{\e^M_n},
\end{align*}
where $C_4 = 2C_3\kappa + C_2$ and is independent of $\Delta t$. Similarly to in the previous section, the error between the continuous quantities $Z^i_t$ and $\Delta Y^i_t$ and their neural network approximations is made up of the above quantities and the discrete approximation error \cref{eq:disc-approx-err}, so the final error for these auxiliary processes is given by
\begin{align*}
\sum_{n=0}^{M-1} \max_{i \in \I}\int_{t_n}^{t_{n+1}} \E ||Z^i_t - \hat\ZZ^i_n(X^\pi_n)||^2 dt \leq& C_4\sum_{n=0}^{M-1}\bm{\e^M_n} + \revised{C^\e M^{-1/2 + \e}}, \\
 \sum_{n=0}^{M-1} \max_{i \in \I} \int_{t_n}^{t_{n+1}} \E |\Delta Y^i_t - \widehat{\Delta \YY}^i_n(X^\pi_n)|^2 dt \leq& C_4\sum_{n=0}^{M-1}\bm{\e^M_n} + \revised{C^\e M^{-1/2 + \e}}. 
\end{align*}


\subsection{Proof of \cref{thm:strat}} 
\label{sec:strat-proof}
We show that the expected payoff of the switching strategy generated by the neural network converges to that of the true optimal strategy.

\begin{remark}
We introduce new notation for this subsection, which explicitly highlights the dependence of the neural network output on the number of time steps, $M$. Specifically we redefine $\hat \YY^i_n(X^\pi_n)$ as $\hat \YY^i_{n,M}(X^\pi_n)$ and $\hat y^i_n(X^\pi_n)$ as $\hat y^i_{n,M}(X^\pi_n)$. 
\end{remark}

Recall that $ \bm{a^{NN,M}}$ is the strategy produced by the OSJ algorithm when the number of time steps is given by $M$.
By definition, $ \bm{a^{NN,M}}$ is a discrete switching strategy where switches can only occur at times $t_n$ for $n = 0, \ldots,M$ and $\bm{a^{NN,M}}$ takes on the value $\alpha^{NN,M,i}_n$ (shortened hereafter to $\alpha_n$) on the interval $[t_n, t_{n+1})$. If the system was in mode $\alpha_{n-1}$ right before time $t_n$, then at time $t_n$ the neural network value function satisfies the relationship
\begin{equation}
\label{eq:strat_evol}
\hat \YY^{\alpha_{n-1}}_{n,M}(X^\pi_n) = \tilde \YY^{\alpha_n}_{n,M}(X^\pi_n) -  C_{\alpha_{n-1}, \alpha_n}(X^\pi_n).
\end{equation}
\begin{remark}
Note that it will often be the case that $\alpha_{n-1} = \alpha_n$, in which case $C_{\alpha_{n-1},\alpha_n} = 0$ by \cref{assm_cost} and no switching cost is incurred. We investigate this further in \cref{lem:S_M}.
\end{remark}

The OSJ-generated strategy produces an expected payoff $J(0,x_0, i, \bm{a^{NN,M}})$ when starting in initial mode $i$ and initial state $x_0$. Let us now consider the error at some arbitrary time $t_n$ between the expected payoff using the neural network strategy and the true value function. For any $n = 0, 1, 2, \ldots, M$, we can use the triangle inequality to yield
\begin{equation*}
|Y^{\alpha_n}_{t_n} - J(t_n, X_{t_n}, \alpha_n,  \bm{a^{NN,M}})| \leq |Y^{\alpha_n}_{t_n} - \hat \YY^{\alpha_n}_{n,M}(X^\pi_n)| + |\hat \YY^{\alpha_n}_{n,M}(X^\pi_n) - J(t_n, X_{t_n}, \alpha_n,  \bm{a^{NN,M}})|.
\end{equation*}
We can also apply \cref{eq:strat_evol} to produce
\begin{align*}
\hat \YY^{\alpha_n}_{n,M} - J(t_n, X_{t_n}, \alpha_n,  \bm{a^{NN,M}}) = &\tilde \YY^{\alpha_n}_{n,M} -  C_{\alpha_{n-1}, \alpha_n}(X^\pi_n) - J(t_n, X_{t_n}, \alpha_n,  \bm{a^{NN,M}})\\
= &  (\tilde \YY^{\alpha_n}_{n,M}(X^\pi_n) - \hat y^{\alpha_n}_{n,M}(X^\pi_n)) \\
&+ (\hat y^{\alpha_n}_{n,M}(X^\pi_n) -  C_{\alpha_{n-1}, \alpha_n}(X^\pi_n) - J(t_n, X_{t_n}, \alpha_n,  \bm{a^{NN,M}})).
\end{align*}
We can use \cref{eq:J-payoff} to express the expected payoff $J(t_n, X_{t_n}, \alpha_{n-1},  \bm{a^{NN,M}})$ recursively as
\begin{align*}
J(t_n, X_{t_n}, \alpha_{n-1},  \bm{a^{NN,M}}) = \E\bigg[&\int_{t_n}^{t_{n+1}} f_{\alpha_n}(s, X_s) ds - C_{\alpha_{n-1}, \alpha_n}(X_{t_n}) \\
&+ J(t_{n+1}, X_{t_{n+1}}, \alpha_n,  \bm{a^{NN,M}})\Big| \F_n \bigg].
\end{align*}
Therefore, recalling the definition \cref{eq:hat_y} and taking the expectation of the absolute value of the difference between $\hat \YY^{\alpha_{n-1}}_{n,M}$ and $J(t_n, X_{t_n}, \alpha_{n-1},  \bm{a^{NN,M}})$ yields
\begin{align*}
\E|\hat \YY^{\alpha_{n-1}}_{n,M} - J(t_n, X_{t_n}, \alpha_{n-1},  \bm{a^{NN,M}})| \leq& \E|\tilde \YY^{\alpha_n}_{n,M}(X^\pi_n) - \hat y^{\alpha_n}_{n,M}(X^\pi_n)| \\
&+ \E|\hat \YY^{\alpha_n}_{n+1, M}(X^\pi_{n+1})  - J(t_{n+1}, X_{t_{n+1}}, \alpha_n,  \bm{a^{NN,M}})|\\
&+ \E|C_{\alpha_{n-1}, \alpha_n}(X_{t_n}) -  C_{\alpha_{n-1}, \alpha_n}(X^\pi_n)| \\
&+ \E\Big[\int_{t_n}^{t_{n+1}} |f_{\alpha_n}(t_n, X^\pi_n) - f_{\alpha_n}(s, X_s)| ds\Big].
\end{align*}
Taking the sum over $n = 0, \ldots , M-1$ on both sides gives us
\begin{align*}
\sum_{n=0}^{M-1} \E|\hat \YY^{\alpha_{n-1}}_{n,M} - J(t_n, X_{t_n}, \alpha_{n-1},  \bm{a^{NN,M}})| \leq& \sum_{n=0}^{M-1} \E|\tilde \YY^{\alpha_n}_{n,M}(X^\pi_n) - \hat y^{\alpha_n}_{n,M}(X^\pi_n)| \\
&+ \sum_{n=0}^{M-1} \E|\hat \YY^{\alpha_n}_{n+1,M}(X^\pi_{n+1})  - J(t_{n+1}, X_{t_{n+1}}, \alpha_n,  \bm{a^{NN,M}})|\\
&+ \sum_{n=0}^{M-1}\E|C_{\alpha_{n-1}, \alpha_n}(X_{t_n}) -  C_{\alpha_{n-1}, \alpha_n}(X^\pi_n) |\\
&+ \sum_{n=0}^{M-1} \E\Big[\int_{t_n}^{t_{n+1}} |f_{\alpha_n}(t_n, X^\pi_n) - f_{\alpha_n}(s, X_s)| ds \Big].
\end{align*}
By subtracting $\sum_{n=1}^{M-1}  \E|\hat \YY^{\alpha_n}_{n,M} - J(t_n, X_{t_n}, \alpha_n,  \bm{a^{NN,M}})|$ from both sides, we obtain
\begin{align*}
|\hat \YY^i_{0,M}(x_0) - J(0, x_0, i,  \bm{a^{NN,M}})| \leq& \sum_{n=0}^{M-1} \E|\tilde \YY^{\alpha_n}_{n,M}(X^\pi_n) - \hat y^{\alpha_n}_{n,M}(X^\pi_n)|\\
&+ \E|\hat \YY^{\alpha_{M-1}}_{M,M}(X^\pi_M)  - J(T, X_T, \alpha_{M-1},  \bm{a^{NN,M}})|\\
&+ \sum_{n=0}^{M-1}\E|C_{\alpha_{n-1}, \alpha_n}(X_{t_n}) -  C_{\alpha_{n-1}, \alpha_n}(X^\pi_n) |\\
&+ \sum_{n=0}^{M-1} \E\Big[\int_{t_n}^{t_{n+1}} |f_{\alpha_n}(t_n, X^\pi_n) - f_{\alpha_n}(s, X_s)| ds \Big].
\end{align*}
We make a few comments at this stage. First, at the terminal time $T = t_M$, $\hat\YY^i_{M,M}(x) \equiv J(T, x, i, \cdot) \equiv g^i(x), \forall x \in \R^d, i \in \I$. In addition, $g^i$ is Lipschitz for all $i \in \I$. 
Therefore, the error at the final step $M$ has error at most $O(M^{-1})$ from the error between the discrete approximation $X^\pi_M$ and the true random variable $X_T$. 
Second, notice that $|C_{\alpha_{n-1}, \alpha_n}(X_{t_n}) -  C_{\alpha_{n-1}, \alpha_n}(X^\pi_n)| \neq 0$ if and only if $\bm{1}_{\{\alpha_{n-1} \neq \alpha_n\}} = 1$. This property, along with the inequality $||\mathbf{x}||_1 \leq \sqrt{M} ||\mathbf{x}||_2, \forall \mathbf{x} \in \R^M$, gives us
\begin{align*}
|\hat \YY^i_{0,M}(x_0) - J(0, x_0, i,  \bm{a^{NN,M}})| \leq& \sum_{n=0}^{M-1} \E|\tilde \YY^{\alpha_n}_{n,M}(X^\pi_n) - \hat y^{\alpha_n}_{n,M}(X^\pi_n)|+ O(M^{-1})\\
+& \sum_{n=0}^{M-1}\E\big[|C_{\alpha_{n-1}, \alpha_n}(X_{t_n}) -  C_{\alpha_{n-1}, \alpha_n}(X^\pi_n) |\bm{1}_{\{\alpha_{n-1} \neq \alpha_n\}}\big]\\
+& \sum_{n=0}^{M-1} \E\Big[\int_{t_n}^{t_{n+1}} |f_{\alpha_n}(t_n, X^\pi_n) - f_{\alpha_n}(s, X_s)| ds \Big]\\
\leq& \sum_{n=0}^{M-1} \E|\tilde \YY^{\alpha_n}_{n,M}(X^\pi_n) - \hat y^{\alpha_n}_{n,M}(X^\pi_n)|+ O(M^{-1})\\
+& \revised{\Big[M\sum_{n=0}^{M-1}\left(\E\big[|C_{\alpha_{n-1}, \alpha_n}(X_{t_n}) -  C_{\alpha_{n-1}, \alpha_n}(X^\pi_n) |\bm{1}_{\{\alpha_{n-1} \neq \alpha_n\}}\big]\right)^2\Big]^{1/2}}\\
+& \sum_{n=0}^{M-1} \E\Big[\int_{t_n}^{t_{n+1}} |f_{\alpha_n}(t_n, X^\pi_n) - f_{\alpha_n}(s, X_s)| ds \Big].
\end{align*}
To proceed, we recall the Lipschitz assumption on the running cost $f$ given by \cref{assm_fg} and the assumption of a strong order 1.0 discrete approximation of $X_t$, as well as the result \cref{eq:loss-bd2}. Then, the inequality can be further simplified as
\begin{align*}
|\hat \YY^i_{0,M}(x_0) - J(0, x_0, i,  \bm{a^{NN,M}})| \leq& \sum_{n=0}^{M-1}\big(C_2\bm{\e^M_n}/M\big) + O(M^{-1}) \\
+&\revised{\left[M\sum_{n=0}^{M-1}\left(\E\big[|C_{\alpha_{n-1}, \alpha_n}(X_{t_n}) -  C_{\alpha_{n-1}, \alpha_n}(X^\pi_n) |\bm{1}_{\{\alpha_{n-1} \neq \alpha_n\}}\big]\right)^2\right]^{1/2}.}
\end{align*}
Focusing on the sum of the switching costs, we apply the Lipschitz assumption for the switching costs given by \cref{assm_cost}, Cauchy--Schwartz inequality, and \cref{rmk:disc} to yield
\begin{align}
&\sum_{n=0}^{M-1}\big(\E\big[|C_{\alpha_{n-1}, \alpha_n}(X_{t_n}) -  C_{\alpha_{n-1}, \alpha_n}(X^\pi_n) |\bm{1}_{\{\alpha_{n-1} \neq \alpha_n\}}\big]\big)^2 \nonumber \\
\leq & [C]^2_l  \sum_{n=0}^{M-1}\big(\E\big[||X_{t_n} - X^\pi_n||\bm{1}_{\{\alpha_{n-1} \neq \alpha_n\}}\big]\big)^2 \nonumber\\
  \leq & [C]^2_l   \sum_{n=0}^{M-1}\E\big[||X_{t_n} - X^\pi_n||^2\big] \E[(\bm{1}_{\{\alpha_{n-1} \neq \alpha_n\}})^2] \nonumber\\
 \leq & O(M^{-2}) \times \E\left[\sum_{n=0}^{M-1}\bm{1}_{\{\alpha_{n-1} \neq \alpha_n\}}\right]. \label{eq:cost-bd}
\end{align}
We define a new random variable $S_M$, representing the number of switches following from a given learned strategy neural network strategy $\bm{a^{NN,M}}$. We define it as
\begin{equation*}
S_M := \sum_{n=0}^{M-1} \bm{1}_{\{\alpha_{n-1} \neq \alpha_n\}}, \forall M \in \N,
\end{equation*}
where $\{\alpha_n\}_{n=0}^{M-1}$ are determined according to $\bm{a^{NN,M}}$ and the choice of $\alpha_{-1} = i \in \I$.

\begin{lemma}
\label{lem:S_M}
The expected number of switches incurred by $\bm{a^{NN,M}}$ is bounded by
\begin{equation*}
\E[S_M] \leq O(1) + C_5 \sum_{n=0}^{M-1}\Big[\big(\bm{\e^M_n}/M\big)^{1/2} + \bm{\e^M_n} \Big],
\end{equation*}
where $C_5 = C_2 + d C_3$.
\end{lemma}
\begin{proof}
From \cref{thm:main} and specifically \cref{eq:y-err}, we have shown that $\hat\YY^i_0(x_0)$ should converge to $Y^i_0$ for all $i \in I$ as $M \to \infty$. We will use this to prove that a bound exists for $\E[S_M]$.
Recall that 
\begin{equation*}
\hat \YY^{\alpha_{n-1}}_{n,M}(X^\pi_n) =  \E\left[f_{\alpha_n}(X^\pi_n)\frac{1}{M} - C_{\alpha_{n-1}, \alpha_n}(X^\pi_n) + \hat \YY^{\alpha_n}_{n+1,M}(X^\pi_{n+1}) + \tilde \YY^{\alpha_n}_{n,M}(X^\pi_n) - \hat y^i_{n,M}(X^\pi_n)|\F_n\right].
\end{equation*}
Taking expectation with respect to $\F_0$, summing over $n \in \{0, \ldots, M-1\}$, and rearranging yields
\begin{equation*}
\sum_{n = 0}^{M-1} \E[C_{\alpha_{n-1}, \alpha_n}(X^\pi_n)] =  \frac{1}{M}\sum_{n = 0}^{M-1}\E[f_{\alpha_n}(X^\pi_n)]  - \sum_{n = 0}^{M-1} \E[\tilde \YY^i_{n,M}(X^\pi_n) - \hat y^i_{n,M}(X^\pi_n)] + g^i(X^\pi_M) - \hat \YY^i_{0,M}(x_0).
\end{equation*}
Therefore, from triangle inequality, 
\begin{equation*}
\Big|\sum_{n = 0}^{M-1} \E[C_{\alpha_{n-1}, \alpha_n}(X^\pi_n)]\Big| \leq  \sum_{n = 0}^{M-1}\frac{\E|f_{\alpha_n}(X^\pi_n)|}{M}  + \sum_{n = 0}^{M-1} \E|\tilde \YY^i_{n,M}(X^\pi_n) - \hat y^i_{n,M}(X^\pi_n)| + |\hat \YY^i_{0,M}(x_0)| + |g^i(X^\pi_M)|.
\end{equation*}
Note that switching costs are always positive if $i \neq j$, and specifically $C_{i,j}(x) \geq \epsilon > 0, \forall i \neq j \in \I$ from \cref{assm_cost}. This implies that
\begin{equation*}
\Big|\sum_{n = 0}^{M-1} \E[C_{\alpha_{n-1}, \alpha_n}(X^\pi_n)]\Big| \geq \sum_{n = 0}^{M-1} \epsilon \E[\bm{1}_{\{\alpha_{n-1} \neq \alpha_n\}}] \geq \epsilon \E[S_M].
\end{equation*}
Now it remains to show that this quantity is bounded above by a quantity which converges to zero. Observe $|g^i(X^\pi_M)| < \infty$. We first tackle the boundedness of $|\hat \YY^i_{0,M}(x_0)|$. The solution to the original optimization problem, $V(0, x_0, i)$, always has a finite-valued expectation, meaning that $\E|Y^i_t| < \infty$ for all times $t \in [0,T]$ (though we are only looking at $t=0$ at the moment) and all modes $i\in \I$. Therefore, using \cref{eq:y-err},
\begin{align*}
\E|\hat \YY^i_{0,M}(x_0)| \leq&\E|\hat \YY^i_{0,M}(x_0) - Y^i_0| + \E|Y^i_0| \leq d\E|\hat \YY^i_{0,M}(x_0) - Y^i_0|^2 + \E|Y^i_0|\\
\leq &d C_3\sum_{n=0}^{M-1} \bm{\e^M_n} + O(M^{-1})  + \E|Y^i_0|.
\end{align*}
We can directly use \cref{assm_fg} to bound $\E|f_{\alpha_n}(X^\pi_n)|$. We also note that $\E|X| \leq (\E[|X|^2])^{1/2}$ and apply \cref{eq:loss-bd2} to bound $\E|\tilde \YY^i_{n,M}(X^\pi_n) - \hat y^i_{n,M}(X^\pi_n)|^2 \leq C_2 \bm{\e^M_n} /M$, giving us
\begin{align*}
\epsilon\E[S_M] \leq& \revised{O(M^{-1/2})}  + \sum_{n = 0}^{M-1} \big(C_2\bm{\e^M_n}/M\big)^{1/2} + d C_3\sum_{n=0}^{M-1} \bm{\e^M_n} + \max_{i \in \I} \E|Y^i_0|.
\end{align*}
To obtain the bound stated in the lemma, we set $C_5 = (\sqrt{C_2} + dC_3)/\epsilon$. We also note that $O(1)$ dominates \revised{$O(M^{-1/2})$}, and this concludes the proof.
\end{proof}

We can now use this bound to continue investigating the error accrued by the neural network-produced switching strategy. We return to \cref{eq:cost-bd} and, armed with \cref{lem:S_M}, conclude that the switching costs can be bounded as
\begin{align*}
&\sum_{n=0}^{M-1}\E\big[|C_{\alpha_{n-1}, \alpha_n}(X_{t_n}) -  C_{\alpha_{n-1}, \alpha_n}(X^\pi_n) |\big]\\
\leq &M^{1/2}\left(\sum_{n=0}^{M-1}\big(\E\big[|C_{\alpha_{n-1}, \alpha_n}(X_{t_n}) -  C_{\alpha_{n-1}, \alpha_n}(X^\pi_n) |\bm{1}_{\{\alpha_{n-1} \neq \alpha_n\}}\big]\big)^2\right)^{1/2} \\
\leq&O(M^{-1/2})\left(O(1) + C_5\sum_{n=0}^{M-1}\Big[\big(\bm{\e^M_n}/M\big)^{1/2} + \bm{\e^M_n}\Big]\right)^{1/2}.
\end{align*}
Therefore, returning to our analysis of the error between the neural network value function and the expected payoff, we incorporate the other error elements to yield
\begin{equation*}
|\hat \YY^i_{0,M}(x_0) - J(0, x_0, i,  \bm{a^{NN,M}})| \leq C_2\sum_{n=0}^{M-1}\frac{\bm{\e^M_n}}{M} + O\left( \frac{1}{M^2}\right)\sqrt{O(1) + C_5\sum_{n=0}^{M-1}\Big[\big(\bm{\e^M_n}/M\big)^{1/2} + \bm{\e^M_n}\Big]}.
\end{equation*}
We now must add $|Y^i_0 - \hat \YY^i_{0,M}(x_0)|$ which is bounded by \cref{eq:y-err} and take the mode-wise maximum over $i$ to achieve our final error bound
\begin{align*}
\max_{i \in \I} |V(0, x_0, i)  - J(0, x_0, i,  \bm{a^{NN,M}}) | \leq& \revised{C^\e M^{-1/2 + \e}} + C_2\sum_{n=0}^{M-1}\left[\frac{\bm{\e^M_n}}{M} + \bm{\e^M_n}\right]\\
&+O\left( \frac{1}{M^2}\right)\sqrt{O(1) + C_5\sum_{n=0}^{M-1}\Big[\big(\bm{\e^M_n}/M\big)^{1/2} + \bm{\e^M_n}\Big]}.
\end{align*}
which is consistent with \cref{eq:thm2}.


\section{Conclusion}
\label{sec:concl}
We have developed an algorithm which uses neural networks and the dynamic programming principle to apply a backward-in-time approach to optimal switching problems in high dimensions. The algorithm is also novel in its ability to accommodate high-dimensional state processes with a finite-variational jump component. We have both applied our approach successfully to numerical examples and obtained analytical convergence results. Our preliminary numerical results indicate that our algorithm is able to accommodate state processes with finite-variational jumps when solving optimal switching problems, and specifically problems related to energy markets. The analytical results present the convergence in terms of the neural network approximation errors of the networks, which have been previously proven to converge to zero. 

Tests in lower dimensions where comparisons can be made against probabilistic methods have verified the accuracy of the algorithm, and tests in higher dimensions that are intractable with more traditional approaches have demonstrated that computation time decreases sub-linearly with dimension. Therefore, this algorithm, while not as fast as other methods for low (i.e. two-dimensional) problems, is an excellent candidate for solving high-dimensional optimal switching problems with jumps, as it is impressively robust with respect to the typical slowdowns experienced as a result of the curse of dimensionality. 

In future, we would like to apply our algorithm to a larger class of optimization problems. There are many problems in energy markets and in other real-world applications in which the state variable is control-dependent. If this is the case, then the chosen switching strategy will affect the evolution of other quantities. For example, the price of electricity in one period (which depends on the chosen level of electricity production) might affect demand in the next period. This would reflect consumer behavioral trends, but would introduce new computational challenges. Namely, it would not be possible to simulate paths for the state process before the optimization was calculated, and so the backward-in-time nature of the OSJ algorithm would not be suited to such problems and a new approach would be necessary. 

While renewable energy sources like wind and solar energy do not incur carbon dioxide emission penalties and are technically free to obtain once installed, they still incur costs to bring online and take offline, and experience much more volatility in their realized capacity. Therefore, adding these sources also adds an increased risk of electricity underproduction. In the future, we could develop an algorithm which contains a realistic method of penalizing unmet electricity demand, and therefore allows us to investigate interesting problems involving ``free", renewable energy sources. Other potential new models could include more complex scheduling options (essentially allowing the electricity producer to consider a greater number of modes when scheduling). 

Another direction could be investigating or improving the robustness of these algorithms. Robustness is crucial in energy applications because electricity producers wish to avoid ``worst-case" scenarios, such as demand spikes or failures in the power grid. The solution to a given class of robust switching control problems is characterized in \cite{BayRobustQVI}, and expanded to infinite time-horizon ergodic problems in \cite{BayDyProgVisc}. It could be productive to expand this theoretical framework by designing a neural network-based algorithm for such robust optimal switching problems.

Finally, new convergence results for neural networks are being developed constantly. Another useful extension could be to present convergence results in terms of concrete quantities such as the neural network width, depth, or overall size. Such results have been calculated for specific classes of neural networks, as in \cite{Yarotsky16, Tanielian21}


\section*{Acknowledgments}
We would like to thank the referees and associate editor for their time, energy, and valuable comments which greatly improved the quality of our paper. 

\appendix
\section{Parameters for Numerical Experiments}

\subsection{Parameters used in \cref{sec:ex_Aid}}
\revised{Below are the parameters used for the dynamics of $X_t = [D_t, A^1_t, A^2_t, A^3_t, S^0_t, S^1_t, S^2_t, S^3_t, P_t]^T \in \R^9$, as well as for the running costs $f$ and switching costs $C_{i,j}$ for \cref{sec:ex_Aid}.}
\begin{table}[H]
\centering
\revised{\begin{tabular}{ |c|c| } 
 \hline
 Parameter & Value \\
 \hline
$\alpha$ & $[4,8,8,8]^T$\\
$\beta$ & $\begin{bmatrix} 15 & 0.1 & 0.1 & 0 \\ 0.1 & 0.5 & -0.1 & 0 \\ 0.1 & -0.1 & 0.5 & 0 \\ 0 & 0 & 0 & 0.5 \end{bmatrix}$\\ 
$X_0$ & $[0,0,0,0, 20, 40, 60, 20, 120]^T$ \\
$\mu$ & $\begin{bmatrix} -4 & 0 & 0 & 1 & 0 \\ 0 & 0 & 0 & 0 & 0 \\ 0 & 2 & -1 & 0 & 1 \\ 0 & 0 & 0 & 0 & 0 \\ 0 & 1 & 1 & 1 & -1 \end{bmatrix}$\\
$\Sigma$ & $\frac{1}{100}\begin{bmatrix} 2.5 & 1.25 & 1.25 & 1.25 & 1.25\\ 1.25 & 5 & 1.25 & 1.25 & 1.25\\ 1.25 & 1.25 & 15 & 1.25 & 1.25 \\ 0.25 & 0.25 & 0.25 & 1.5 & 1.25 \\ 1.25 & 1.25 & 1.25 & 1.25 & 3 \end{bmatrix}$\\
$h^0$ & $[0.5, 2, 0]^T$ \\
$h$ & $[1, 1.5, 1.5]^T$  \\
$\bm{c}$ & $[0.1, 0.1, 0.5]^T$\\
$\epsilon$ & 0.001\\
 \hline
\end{tabular}}
\caption{Parameters associated with \cref{sec:ex_Aid}.}
\label{tab:aid_params}
\end{table}


\bibliographystyle{abbrv}
\bibliography{ML-Bib}

\end{document}